\documentclass[11pt,letterpaper]{amsart}
\usepackage[usenames,dvipsnames,svgnames,table]{xcolor}
\usepackage[pagebackref,linktocpage=true,colorlinks=true,linkcolor=Blue,citecolor=BrickRed,urlcolor=RoyalBlue]{hyperref}
\usepackage[alphabetic,msc-links,abbrev]{amsrefs} 

\usepackage{stmaryrd}
\usepackage{mathrsfs}
\usepackage{amsmath}
\usepackage{amssymb}
\usepackage{accents}
\usepackage{amsthm}
\usepackage{mathtools}
\usepackage{graphicx}
\usepackage{enumitem}
\usepackage{bbm}
\usepackage[colorinlistoftodos,prependcaption,textsize=tiny,
textwidth=0.85in]{todonotes}

\usepackage{esint}
\usepackage{stackengine}

\newtheorem{theorem}{Theorem}[section]
\newtheorem{lemma}[theorem]{Lemma}
\newtheorem{proposition}[theorem]{Proposition}
\newtheorem{definition}[theorem]{Definition}
\newtheorem{corollary}[theorem]{Corollary}

\theoremstyle{remark}
\newtheorem{remark}[theorem]{Remark}

\theoremstyle{definition}

\numberwithin{equation}{section}

\def\R{\mathbb{R}}
\def\N{\mathbb{N}}
\def\Sph{\mathbb{S}}
\def\d{\partial}

\def\supp{\textnormal{\textrm{supp}}}
\def\sgn{\textnormal{\textrm{sgn}}}

\def\ep{\varepsilon}

\def\mB{\mathcal{B}}
\def\mC{\mathcal{C}}

\def\mH{\mathcal{H}}

\def\mP{\mathcal{P}}
\def\mR{\mathcal{R}}

\renewcommand{\div}{\textnormal{\textrm{div}}}

\usepackage{scalerel}

\newcommand\ffhat[1]{\arraycolsep=0pt\relax%
	\begin{array}{c}
		\stretchto{
			\scaleto{
				\scalerel*[\widthof{\ensuremath{#1}}]{\kern-.5pt\bigwedge\kern-.5pt}
				{\rule[-\textheight/2]{1ex}{\textheight}} 
			}{\textheight} %
		}{0.5ex}\\           
		#1\\                 
		\rule{-1ex}{0ex}
	\end{array}
}

\usepackage{scalerel,stackengine}
\stackMath
\newcommand\fhat[1]{%
	\savestack{\tmpbox}{\stretchto{%
			\scaleto{%
				\scalerel*[\widthof{\ensuremath{#1}}]{\kern.1pt\mathchar"0362\kern.1pt}%
				{\rule{0ex}{\textheight}}
			}{\textheight}%
		}{2.4ex}}%
	\stackon[-6.9pt]{#1}{\tmpbox}%
}
\usepackage[alphabetic]{amsrefs}
\usepackage{tikz}
\usepackage{tikz-3dplot}

\begin{document}
	
	\title{
		A Monotonicity formula for almost self-similar suitable weak solutions 
		to the stationary Navier-Stokes equations in $\mathbb R^5$%
	}
	\author{
		Yucong {\sc Huang} and Aram  {\sc Karakhanyan} 
	}
	\email{aram6k@gmail.com}
	\date{}
	\maketitle
\begin{abstract}
In this paper we show that a suitable weak solution to the 
stationary Navier-Stokes system in $\R^5$, cannot behave like a
self-similar function of degree negative one  
if the lower limit of the local Reynolds number is finite.

To prove the result we develop a method  that uses a 
monotonicity formula approach, classification of homogenous solutions to the 
incompressible Euler equations in $\R^5$, and a projection theorem. 
\end{abstract}	

\section{Introduction}

In this paper we study the local behavior of the weak solutions of the stationary incompressible 
Navier-Stokes equations in five space dimensions
\begin{equation}\label{eq:problem}
\left.\begin{array}{rrr}
u^ju^i_j+P_i=\Delta u^i, \quad i=1, 2, 3, 4, 5,\\
\div u =0,
\end{array}
\right\}
\quad\mbox{in}\ \Omega\subset \R^5
\end{equation}
where $\Omega\subset \R^5$ is a domain . 

The existence of weak solutions under various assumptions on the 
boundary data and $\Omega$ has been established in \cite{Galdi}, \cite{FR-Pisa}, \cite{Struwe-per}.  
Moreover, in  \cite{FR-Pisa}, \cite{FR-arma}, \cite{Struwe-per} the authors constructed 
smooth solutions of \eqref{eq:problem-3d}.

The problem \eqref{eq:problem} has a number of similarities with the 
dynamic Navier-Stokes system in three space dimensions
\begin{equation}\label{eq:problem-3d}
\left.\begin{array}{rrr}
u_t^i+u^ju^i_j+P_i=\Delta u^i, \quad i=1, 2, 3,\\
\div u =0,
\end{array}
\right\}
\mbox{in}\ \Omega\times(0, T), \Omega\subset \R^3.
\end{equation}
For instance, in both cases $u\in L^{\frac{10}3}_{loc}, P\in L^{\frac 53}_{loc}, \nabla P\in L^{\frac54}_{loc}$, see  \cite{Seregin-book}. 
Due to this a number of mathematicians studied the stationary 
Navier-Stokes equations in higher dimensions in order to develop 
stronger analytical methods which may be applicable to the 
dynamic case \eqref{eq:problem-3d}, see \cite{Galdi}.

In this context, of particular interest is the problem of estimating the 
dimension of the singular points of suitable weak solutions $u$, i.e. the points where 
$u$ is not bounded. Scheffer \cite{Scheffer}  
proved such results  for \eqref{eq:problem-3d} and later
Caffarelli, Kohn and Nirenberg \cite{CKN} improved upon it 
showing that  the Hausdorff dimension of the singular set in space-time is atmost one. 
We note that the latter result can be established by a different method
by looking at the   
small perturbations of the Stokes system \cite{Lin}. 
For \eqref{eq:problem} the partial regularity is proved in \cite{Struwe}.

At the possible singular point $(x_0, t_0)$ the scale invariance  $u(x, t)\mapsto ru(x_0+rx, t_0+r^2), r>0$ suggests that at the scale $r$, $u$ behaves like $1/r$ near $(x_0, t_0)$.
A natural question that follows from this observation is whether 
one can classify the scale invariant solutions.
This has been the main approach towards understanding the structure of possible singularities.
\v{S}ver\'{a}k's classification for the self-similar solutions \cite{Sverak} for Navier-Stokes equations \eqref{eq:problem} shows that 
a solution of the form $h(x)=\frac{\zeta(\frac x{|x|})}{|x|}$, with some smooth vectorfield $\zeta$, must be identically zero. 

Another questions following from this result is whether the solutions sufficiently close to the self-similar one 
are in fact zero. It is easy to see that the self-similar vectorfields $h=\frac{\zeta(\frac x{|x|})}{|x|}$ form a Hilbert subspace $\mH(R)$ of 
the Sobolev space $\mathcal W^{1, 2}(B_R)$ in $\R^5$ with an appropriately scaled invariant norms,
\begin{equation}\label{mH1}
	\mH(R):= \left\{ h\in W^{1,2}(B_R) \,\bigg\vert\, \exists \zeta \in W^{1,2}(\mathbb{S}^{N-1}): \ h(x)=\frac{1}{|x|}\zeta\big( \frac{x}{|x|} \big) \ \text{for $x\in B_R$ a.e.} \right\}.
\end{equation}
Thus the Hilbert projection theorem yields that $\mathcal W^{1,2}(B_R) = \mH(R) \oplus \mH(R)^{\perp}$ and we can define the the corresponding projection operator as $\mathcal{P}_R[\,\cdot\,]: \mathcal W^{1,2}(B_R)\to \mH(R)$. Using this, we can measure the error $u-\mathcal{P}_R[u]$ in terms of the $W^{1,2}$ norm of $u$. 

Our work is motivated by the following question. 
 
{\textit{  If a suitable weak solution to the stationary Navier-Stokes system develops a singularity, 
 can it asymptotically  become self-similar?}}
 
Our  main result in this direction can be stated as follows: 
\begin{theorem}\label{TH-00}
Let $u$ be a suitable weak solutions. Suppose that 
the following two conditions hold:
\begin{equation}
\liminf_{r\to 0}M(r):=\liminf_{r\to 0}\int_{B_r}\!\! \left( \frac{|u|^2}{r^3} + \frac{|\nabla u|^2}{r} \right) <\infty, \quad \liminf_{r\to 0}\frac1{r^2}\int_{B_r}\!\!(|u|^2+2P)u\cdot\frac x{|x|}>0.
\end{equation}
Then  for any $\{r_k\}_{k=1}^\infty, r_k\downarrow 0$ there is a subsequence $r_{k_m}$ such that 
the scaled solutions $u_{r_{k_m}}(x)=r_{k_m}u(r_{k_m}x)$, converge to 
a homogenous vector field of degree negative one, and hence 
$x=0$ is a regular point.
\end{theorem}
The proof of Theorem \ref{TH-00} uses the monotonicity 
formula introduced in Proposition \ref{prop:dAdr}, and 
a scaling argument. See Lemma \ref{lem:positiveinf} for the proof.
Note that there are no smallness assumptions in the 
statement of Theorem \ref{TH-00}.

If $\liminf_{r\to 0}M(r)<\infty$
then the singularity may occur only if the function 
\begin{equation}
\text{\Large$\wp$}(r)  =\frac{1}{r^2} \int_{B_r}  \Big\{{|u|^2} + 2P \Big\} (u\cdot \frac{x}{|x|})  \, dx .
\end{equation}
takes nonpositive values as $r\to 0$. Moreover, 
if $u$ is of the form $\frac{\zeta(\frac x{|x|})}{|x|}$ then one can check that 
$\text{\Large$\wp$}(r) =0$. This observation motivates the formulation of 
a condition in our next result that allows to control $\text{\Large$\wp$}(r)$.  
\begin{theorem}\label{TH}
Let $u$ be a suitable weak solution of \eqref{eq:problem}, $B_1\subset \Omega$, and
\[
m:=\liminf_{R\to 0}M(R)<\infty \qquad \text{where} \quad M(R):=\int_{B_R} \Big( \frac{|u|^2}{R^3} + \frac{|\nabla u|^2}{R} \Big).
\] Let $\mP_R[\,\cdot\,]: \mathcal W^{1,2}(B_R)\to \mH(R)$ be the projection operator for the space \eqref{mH1}. There exists $\ep(m)>0$ 
such that 
if  
\begin{equation}\label{eq:13-alt}
\frac{1}{R^3} \int_{B_R} \left| u - \mP_R[u] \right|^2 + \frac{1}{R} \int_{B_R} \left| \nabla u - \nabla \mP_R[u] \right|^2 \le  \ep(m) M(R)
\end{equation}
holds
for all $R\in (0, R_0)$ then $u$ is regular at $x=0$.  
\end{theorem}

It is known that if $u\in \mH$ then $u=0$ \cite{Sverak}. In this context, Theorem \ref{TH} 
states that if $\|u-\mP_R[u]\|_{\mathcal{W}^{1,2}(R)}$ is small compared to $\|u\|_{\mathcal{W}^{1,2}(R)}$ then 
$\text{\Large$\wp$}(r)$ is smaller than $M^{\frac32}(r)$, which 
after application of Proposition \ref{prop:dAdr} implies that 
$u=0$.

As opposed to the main result in \cite{CKN}, we do not assume that $u$ is small in some scale invariant seminorm, reminiscent to  
the ``local" Reynolds number $\frac1{r}\int_{B_r}|\nabla u|^2$. This  leads us to the classification of the 
self-similar solution of the incompressible Euler equations in $\R^5$. 
In fact, we prove that for such solutions 
the Bernoulli pressure is zero. This is the first  key point in our proof of the main technical result, Proposition \ref{prop:cubicEst}.

The second key point is the construction of a monotonicity formula for the suitable weak solutions, which follows from 
the weak energy inequality. 

We compare Theorem \ref{TH} with the well-known 
regularity criteria for suitable weak solutions of 
\eqref{eq:problem-3d}, which in its most general form, can be stated as follows:
let $Q(R)=B_R\times(-R^2, 0)$ and define the local Reynolds numbers
\[
E(R)=\frac1R\int_{Q(R)}|\nabla u|^2, \quad C(R)=\frac1{R^2}\int_{Q(R)}|u|^3.
\]


Then the following statement holds: for every $M>0$ there is 
$\ep(M)>0$ such that 
$\limsup_{R\to 0} C(R)<M, \liminf_{R\to 0}E(R)<\epsilon(M)$ imply that the origin is a regular point.

This result can be found in 
Seregin's paper \cite{Seregin}, Theorem 1.4. Note that
 $\limsup_{R\to 0} C(R)<M$ implies that there is a constant $C_0(M)$ 
 depending on $M$ such that $\limsup_{R\to 0} E(R)<C_0(M)$ \cite{Lewis}.

As opposed to this result, we do not impose the finiteness of upper limit 
of the Reynolds number. Instead, we assume that the lower limit is finite, i.e.
\[
\liminf_{R\to 0} M(R)<\infty.
\]

This is the replacement of the condition $\limsup_{R\to 0} C(R)<M$ in \cite{Seregin}.

As for the other condition, $\liminf_{R\to 0}E(R)<\epsilon(M)$, it is replaced by closeness assumption: more precisely, we assume that there is a vectorfield $h$  
homogeneous of degree negative one such that 
$u-h$ has a suitable small norm compared to $M(R)$. 
See Section \ref{sec:energy}
for precise definitions. 
Hence our conditions are weaker.

The paper is organized as follows:
In Section \ref{sec:mon} we introduce one of our main technical tools, the 
monotonicity formula and prove Theorem \ref{TH-00}. In the next section 
we classify the self-similar solutions of the incompressible Euler equations in 
$\mathbb R^5$, and prove that for such solutions the Bernoulli 
pressure  is zero. 
In Section \ref{sec:energy} we prove some estimates for the pressure. 
Section \ref{sec:technical} contains one of our main estimates of the 
cubic term that appears in the local energy inequality.  
In order to control the growth of $M(R)$ we prove an iteration result in Section \ref{sec:iteration}, and apply it to obtain a local bound in Section 
\ref{sec:sup-est}.
The proof of Theorem \ref{TH} is given in Section.
\ref{sec:proofofmain}.
We also added an appendix at the end of the paper that contains 
some estimates and computations used in the proof of Theorem \ref{TH}.

\section{Notations}
We fix some notation that will be used throughout the paper. 
\begin{enumerate}[label=(\arabic*),ref=(\arabic*)]
	\item For $R>0$, we set $B_R\vcentcolon= \{ x\in \R^5 \, \vert \, |x|\le R \}$ and $B\vcentcolon= B_1$. For function $f\vcentcolon B_R \to \R$, we denote
	\begin{equation*}
		[f]_R \vcentcolon= \frac{1}{|B_R|} \int_{B_R} f, \qquad [f]\vcentcolon= [f]_1.
	\end{equation*}
	\item For $R>0$ and function, we define the functional
	\begin{equation}\label{MR}
		M[u](R) \vcentcolon= \int_{B_R} \Big( \frac{|u|^2}{R^3} + \frac{|\nabla u|^2}{R} \Big).
	\end{equation}
	Moreover, if the choice of function $u$ is unambiguous then we also use the abbreviated notation 
	\begin{equation*}
		M(R) \vcentcolon= M[u](R).
	\end{equation*}
\end{enumerate}


\section{Monotonicity formula for the stationary Navier-Stokes system}\label{sec:mon}
Let $(u, P)$ be a stationary solution to the Navier-Stokes equations:
\begin{subequations}\label{st-NS-N}
	\begin{align}
		& \div u = 0 && \text{for } \ x\in\R^N, \label{st-NS1}\\
		& (u\cdot \nabla) u + \nabla P = \Delta u && \text{for } \ x\in\R^N, \label{st-NS2}
	\end{align}
	where $u \in \mathcal W^{1,2 }_{loc}(\R^N)$.
\end{subequations}

Given a weak solution $(u,P)$, we set the \textbf{energy defect measure} $\mu \vcentcolon \mC_{c}^{\infty}(\R^N) \to \R$ as
\begin{align}\label{mu}
	\mu(\phi) \vcentcolon= \int_{\R^N}\!\! \Big\{ \big( \frac{|u|^2}{2} + P \big) (u\cdot \nabla) \phi + \frac{|u|^2}{2} \Delta\phi - \phi |\nabla u|^2 \Big\}\, d x\ge 0, \qquad \text{for } \ \phi \in \mC_{c}^{\infty}(\R^N). 
\end{align}

\begin{definition}
	A weak solution $(u,P)$ is defined to be a \textbf{suitable weak solution} of \eqref{st-NS-N}, if there exits a measure $\mu$ such that 
	\begin{equation}\label{LEI}
		\Delta \frac{|u|^2}{2} - |\nabla u|^2 - \div \Big\{ \big(\frac{|u|^2}{2} + P\big) u \Big\} = \mu \quad \text{in the sense of distribution.}
	\end{equation}
\end{definition}
\begin{lemma}
For $u^\circ\in \mathcal W^{2,2}(\Omega)$, with small norm,  and $\Omega$ a bounded domain with Lipschitz 
boundary there is a suitable weak solution to the problem 
\begin{align*}
\left.\begin{array}{rrr}
u \cdot \nabla u+\nabla P=\Delta u, \\
\div u =0,
\end{array}
\right\}
\quad\mbox{in}\ \Omega\subset \R^5, \\
u=u^\circ\quad \mbox{on} \ \partial \Omega.
\end{align*}
\end{lemma}
\begin{proof}
The existence of a suitably weak solutions with a boundary condition $u^\circ\in  \mathcal W^{2,2}(\Omega)$ on $\partial \Omega$ follows from a standard approximation argument: suppose that 
$\Omega$ is bounded and $\partial \Omega$ is smooth, 
$\rho_\epsilon,\epsilon>0$ is the standard mollifier, then we consider the following problem
\begin{align*}\label{eq:problem-k8}
\left.\begin{array}{rrr}
((u*\rho_\epsilon)\cdot \nabla u)+\nabla P=\Delta u, \\
\div u =0,
\end{array}
\right\}
\quad\mbox{in}\ \Omega\subset \R^5, \\
u=u^\circ\quad \mbox{on} \ \partial \Omega.
\end{align*}
We can write $u=v+u^\circ$, and reduce the problem to 
homogeneous boundary condition for $v$, which now solves the 
system
\begin{align*}
((v*\rho_\epsilon)\cdot \nabla v)+ ((u^\circ*\rho_\epsilon)\cdot \nabla v)+((v*\rho_\epsilon)\cdot \nabla u^\circ)+\nabla P=\Delta v+\Delta u^\circ- ((u^\circ*\rho_\epsilon)\cdot \nabla u^\circ), 
\\
\div v =0,
\end{align*}
in $\Omega$.
It follows from 
Galerkin's method \cite{Galdi} that there is a 
weak solution $v_k \in  \mathcal W^{2,2}(\Omega)$ 
of the problem

\begin{equation}\label{eq:problem-k}
\left.\begin{array}{rrr}
((v_k*\rho_\epsilon)\cdot \nabla v_k)+ ((u^\circ*\rho_\epsilon)\cdot \nabla v_k)+((v_k*\rho_\epsilon)\cdot \nabla u^\circ)+\nabla P_k=\Delta v_k+f^\circ, \\
\div v_k =0,
\end{array}
\right\}
\quad\mbox{in}\ \Omega\subset \R^5, 
\end{equation}
where $u_k$ belongs to the span of the first 
$k$ functions of the countable basis in $ \mathcal W^{1, 2}(\Omega)$ of smooth divergence free vectorfields 
$\{\phi_m\}_{m=1}^\infty$ vanishing on $\partial \Omega$, and 
$f^\circ=\Delta u^\circ- ((u^\circ*\rho_\epsilon)\cdot \nabla u^\circ)$. 
Note that 
\[
\int_{\Omega} ((v_k*\rho_\epsilon)\cdot \nabla v_k)\phi_l=-
\int_{\Omega} ((v_k*\rho_\epsilon)\otimes v_k)\div \phi_l=0, \quad l=1, \dots, k.
\]
Consequently 
\[
\int_{\Omega}|\nabla v_k|^2\lesssim \int_\Omega |v_k|^2|\nabla(u^\circ*\rho_\epsilon)|
+
|((v_k*\rho_\epsilon)\cdot \nabla u^\circ)||v_k|
+
|f^\circ v_k|.
\]
See \cite{Galdi} Lemma IX.3.2. and Theorem IX.4.1 and Remark IX.4.10.
Therefore, under suitable assumptions on $\|u^\circ\|_{W^{2,2}}$ we obtain the 
uniform estimate $\int_{\Omega}|\nabla v_k|^2\le C(\|u^\circ\|_{ \mathcal  W^{2,2}})$.

Moreover, the solutions $u_k=u^\circ+v_k\in \mathcal W^{2, 2}(\Omega)$, and hence $u_k\psi$
is an admissible test function in the weak formulation of the equation, implying 
\begin{equation}
		\int |\nabla u_k|^2 \psi
		\le 
		\int \left(-u_k \nabla u_k +(|u_k|^2+2P_k)u_k\right)\cdot \nabla \psi.
	\end{equation}
Thus the existence of a suitable weak solution follows from a 
standard compactness argument, by first letting $k\to \infty$ for a fixed $\epsilon$, and then 
$\epsilon \to 0$. 
\end{proof}

\begin{proposition}\label{prop:dAdr}
	Suppose $N=5$ and $(u,P)$ is a suitable weak solution of \eqref{st-NS-N}. For $r>0$, define
	\begin{equation}
		\left\{
		\begin{aligned}
			&D(r)\vcentcolon=   \int_{B_r} \Big\{ \frac{15}{4r^3}|u|^2 + \frac{1}{4r}|\nabla u |^2 + \frac{3}{4r^3}\left| \nabla (|x|u) \right|^2 + \frac{3(r^2-|x|^2)}{4r^3} |\nabla u|^2\Big\}\, dx,\\
			&A(r) \vcentcolon= \frac{1}{r^3}\int_{B_r}(x\cdot \nabla)\frac{|u|^2}{2}\, dx + \frac{9}{4r^3}\int_{ B_r}|u|^2 \, dx  - \frac{1}{r^{2}} \int_{B_r} \big\{ \frac{|u|^2}{2} + P \big\} u\cdot \frac{x}{|x|}  \, dx .
		\end{aligned}
		\right. 
	\end{equation}
	Then the following differential equation holds for $r>0$,
	\begin{equation*}
		\frac{d A}{d r} \ge \frac{1}{r} D(r) + \frac{2}{r^3} \int_{B_r}  \Big\{\frac{|u|^2}{2} + P \Big\} (u\cdot \frac{x}{|x|})  \, dx .
	\end{equation*}
\end{proposition}
\begin{proof}
	Let us consider the function 
	\begin{equation}
		\phi(x)=
		\left\{
		\begin{array}{lll}
			1 \quad & \mbox{if}  \ |x|<r,\\
			\frac{r-|x|}\epsilon \quad &\mbox{if} \ r\le |x|\le r+\epsilon,\\
			0 \quad & \mbox{if}\ |x|>r+\epsilon.\\
		\end{array}
		\right.
	\end{equation}
	We then mollify this function and take $\psi=\phi*\rho_\delta$, 
	where $\rho$ is the mollification kernel. 
	Note that $\psi\in  \mathcal C_c^\infty(\R^N)$.
	We use $\psi$ as a test function in the local 
	energy inequality  to obtain 
	\begin{equation}
		\int |\nabla u|^2 \psi
		\le 
		\int \left(-u \nabla u +(|u|^2+2P)u\right)\cdot \nabla \psi.
	\end{equation}

	For fixed $\epsilon$, let $\delta\to 0$. Using Lebesgue's theorem we obtain   
	
	\begin{equation}
		\int |\nabla u|^2 \phi
		\le 
		\int \left(-u \nabla u +(|u|^2+2P)u\right)\cdot \nabla \phi.
	\end{equation}
	Note that 
	\begin{equation}
		\nabla \phi(x)=
		\left\{
		\begin{array}{lll}
			0 \quad & \mbox{if}  \ |x|<r,\\
			-\frac1\epsilon\frac{x}{|x|} \quad &\mbox{if} \ r\le |x|\le r+\epsilon,\\
			0 \quad & \mbox{if}\ |x|>r+\epsilon,\\
		\end{array}
		\right.
	\end{equation}
	hence for this choice of the test function the energy inequality 
	takes the following form 
	
	\begin{equation}
		\int_{B_{r+\epsilon}} |\nabla u|^2 \phi
		\le 
		\int_{B_{r+\epsilon}\setminus B_{r}} -\frac1\epsilon\left(-u \nabla u +(|u|^2+2P)u\right)\cdot \frac{x}{|x|}.
	\end{equation}
	Integrate over the interval $r\in [a, b]$ to get 
	
	\begin{align*}
		\int_a^b dr\int_{B_{r+\epsilon}} |\nabla u|^2 \phi
		&\le 
		\int_a^b dr\int_{B_{r+\epsilon}\setminus B_{r}} -\frac1\epsilon\left(-u \nabla u +(|u|^2+2P)u\right)\cdot \frac{x}{|x|}\\
		&=
		\int_a^b dr\int_{B_{r+\epsilon}} -\frac1\epsilon\left(-u \nabla u +(|u|^2+2P)u\right)\cdot \frac{x}{|x|} \\
		&-\int_a^b dr\int_{B_{r}} -\frac1\epsilon\left(-u \nabla u +(|u|^2+2P)u\right)\cdot \frac{x}{|x|}.
\end{align*}		
Substituting $r=s-\epsilon, s\in [a+\epsilon, b+\epsilon]$ in the first integral yields
\begin{align*}
\int_a^b dr\int_{B_{r+\epsilon}} |\nabla u|^2 \phi
&\leq\int_{a+\epsilon}^{b+\epsilon}dr\int_{B_r} -\frac1\epsilon\left(-u \nabla u +(|u|^2+2P)u\right)\cdot \frac{x}{|x|}\\
		&-
		\int_{a}^{b}dr\int_{B_r} -\frac1\epsilon\left(-u \nabla u +(|u|^2+2P)u\right)\cdot \frac{x}{|x|}\\
		&=-\int^{a+\epsilon}_{a}dr\int_{B_r} -\frac1\epsilon\left(-u \nabla u +(|u|^2+2P)u\right)\cdot \frac{x}{|x|}\\
		&+
		\int^{b+\epsilon}_{b}dr\int_{B_r} -\frac1\epsilon\left(-u \nabla u +(|u|^2+2P)u\right)\cdot \frac{x}{|x|}.
	\end{align*}
	Since the integrals over $B_r$ are continuous function of $r$, then after applying the mean value 
	theorem, we get  
	\begin{align*}
		\int_a^b dr\int_{B_{r+\epsilon}} |\nabla u|^2 \phi
		&\le 
		\int_{B_{r_1^*(\epsilon)}}  \left(-u \nabla u +(|u|^2+2P)u\right)\cdot \frac{x}{|x|}\\
		&-
		\int_{B_{r_2^*(\epsilon)}} \left(-u \nabla u +(|u|^2+2P)u\right)\cdot \frac{x}{|x|}, 
	\end{align*}
	where $r_1^*(\epsilon)\in [a, a+\epsilon]$ and $r_2^*(\epsilon)\in [b, b+\epsilon]$.
	Letting $\epsilon\to 0$ and using Lebesgue's dominated convergence theorem we infer 
	\begin{align*}
		\int_a^b dr\int_{B_{r}} |\nabla u|^2 
		&\le 
		\int_{B_{a}}  \left(-u \nabla u +(|u|^2+2P)u\right)\cdot \frac{x}{|x|}\\
		&-
		\int_{B_{b}} \left(-u \nabla u +(|u|^2+2P)u\right)\cdot \frac{x}{|x|}.
	\end{align*}
	Taking $b=R+\Delta R, a=R$, we get for almost every $R$, the following inequality 
	\begin{align}\label{distr}
		\int_{B_{R}} |\nabla u|^2 
		&\le
		-\int_{\partial B_{R}} \left(-u \nabla u +(|u|^2+2P)u\right)\cdot \frac{x}{|x|}.
	\end{align}
	
	It is convenient to rewrite \eqref{distr} in the following equivalent form 
	\begin{align}\label{temp1:dAdr}
		\frac{1}{2r} \int_{\d B_r} x\cdot \nabla |u|^2 \, d S_x  - \int_{B_r} |\nabla u|^2 \, d x - \int_{\d B_r} \frac{x}{|x|}\cdot u \big(\frac{|u|^2}{2} + P \big)\, dS_x \ge 0 
	\end{align}
	By divergence theorem, we obtain that
	\begin{align*}
		&
		\frac{1}{2r^{N-2}} \int_{\d B_r} x\cdot \nabla |u|^2 \, d S_x = \frac{1}{2r^{N-2}} \frac{d}{d r} \int_{B_r} x\cdot \nabla |u|^2\, dx\\
		=& \frac{d}{dr} \bigg( \frac{1}{2 r^{N-2}} \int_{B_r} x\cdot \nabla |u|^2 \, dx \bigg) + \frac{N-2}{2r^{N-1}} \int_{B_r} x\cdot \nabla |u|^2 \, dx\\
		=& \frac{d}{dr} \bigg( \frac{1}{2 r^{N-2}} \int_{B_r} x\cdot \nabla |u|^2 \, dx \bigg) + \frac{N-2}{2r^{N-2}} \int_{ \d B_r} |u|^2 \, dS_x - \frac{N(N-2)}{2 r^{N-1}} \int_{B_r} |u|^2 \, dx\\
		=& \frac{d}{d r} \bigg( \frac{1}{2 r^{N-2}} \int_{B_r} x\cdot \nabla |u|^2 \, dx + \frac{N-2}{2 r^{N-2}} \int_{B_r} |u|^2 \, dx \bigg)  - \frac{N-2}{r^{N-1}} \int_{B_r} |u|^2 \, dx \\
		=& \frac{d}{d r} \bigg( \frac{1}{2 r^{N-2}} \int_{B_r} \big\{ N|u|^2 + x\cdot \nabla |u|^2 \big\} \, dx \bigg) - \frac{1}{r^{N-2}}\int_{\d B_r} |u|^2 \, d S_x.
	\end{align*}
	Multiplying equation \eqref{temp1:dAdr} by $r^{3-N}$, then substituting the above identity, we get
	\begin{align}\label{temp2:dAdr}
		&\frac{d}{d r} \bigg( \frac{1}{2 r^{N-2}} \int_{B_r} \big\{ N|u|^2 + x\cdot \nabla |u|^2 \big\} \, dx \bigg)\\
		 \ge & \frac{1}{r^{N-2}}\int_{\d B_r} |u|^2 \, d S_x + \frac{1}{r^{N-3}}\int_{B_r}|\nabla u|^2 \, dx 
		 + \frac{1}{r^{N-2}}\int_{\d B_r} x\cdot u \big(\frac{|u|^2}{2} + P \big)\, dS_x.\nonumber
	\end{align}
	Moreover, we have by completing the square that,
	\begin{align*}
		\int_{B_r}|u|^2 \, dx 
		=&\int_{B_r} \big|u+(x\cdot \nabla) u\big|^2\, dx - \int_{B_r} \big\{ x\cdot \nabla |u|^2 + \big|(x\cdot \nabla) u\big|^2 \big\} \, dx.
	\end{align*}
	Using this, we obtain that
	\begin{align*}
		\frac{N-1}{r^{N-2}}\int_{\d B_r} |u|^2 \, d S_x =&  \frac{1}{r^{N-2}}\frac{d}{dr}\int_{B_r}|u|^2\, d x + \frac{N-2}{r^{N-2}}\int_{ \d B_r}|u|^2 \, d S_x \\
		=& \frac{d}{dr}\bigg( \frac{1}{r^{N-2}} \int_{B_r}|u|^2\, dx \bigg)  +\frac{N-2}{r^{N-1}}\int_{B_r} |u|^2 \, dx + \frac{N-2}{r^{N-2}}\int_{\d B_r} |u|^2 \, dx \\
		=& \frac{d}{dr}\bigg( \frac{1}{r^{N-2}} \int_{B_r}|u|^2\, dx \bigg)  + \frac{N-2}{r^{N-1}}\int_{B_r} \big|u+(x\cdot \nabla)u\big|^2\, dx\\ 
		& - \frac{N-2}{r^{N-1}}\int_{B_r} |(x\cdot \nabla) u|^2\, dx - \frac{N-2}{r^{N-1}}\int_{B_r} x\cdot \nabla |u|^2 \, dx  + \frac{N-2}{r^{N-2}}\int_{\d B_r} |u|^2 \, dx\\
		=& \frac{d}{dr}\bigg( \frac{1}{r^{N-2}} \int_{B_r}|u|^2\, dx \bigg)  + \frac{N-2}{r^{N-1}}\int_{B_r} \big|u+(x\cdot \nabla)u\big|^2\, dx \\ 
		&- \frac{N-2}{r^{N-1}}\int_{B_r} \big|(x\cdot \nabla) u\big|^2\, dx + \frac{N(N-2)}{r^{N-1}}\int_{B_r} |u|^2 \, dx.
	\end{align*}
	Dividing both sides by $(N-1)$ yields the following equation
	\begin{align*}
		\frac{1}{r^{N-2}}\int_{\d B_r} |u|^2 \, d S_x =& \frac{d}{dr}\bigg( \frac{1}{(N-1)r^{N-2}} \int_{B_r}|u|^2\, dx \bigg)  + \frac{N-2}{N-1}\frac{1}{r^{N-1}}\int_{B_r} \big|u+(x\cdot \nabla)u\big|^2\, dx\\ 
		& - \frac{N-2}{N-1}\frac{1}{r^{N-1}}\int_{B_r} \big|(x\cdot \nabla) u\big|^2\, dx + \frac{N(N-2)}{N-1}\frac{1}{r^{N-1}}\int_{B_r} |u|^2 \, dx. 
	\end{align*}
	Substituting the above into equation \eqref{temp2:dAdr}, we have
	\begin{align}\label{temp3:dAdr}
		&\frac{d}{dr} \bigg( \frac{1}{r^{N-2}} \int_{B_r} x\cdot \nabla \frac{|u|^2}{2} + \frac{N^2-N-2}{2(N-1)} \frac{1}{r^{N-2}}\int_{B_r}|u|^2\, dx \bigg)\\
		\ge & 
		\frac{N-2}{N-1} \frac{1}{r^{N-1}} \int_{B_r} \big|u+(x\cdot \nabla)u\big|^2 \, dx + \frac{N-2}{N-1} \frac{1}{r^{N-3}}\int_{B_r} \big\{ |\nabla u|^2 - \big|(x\cdot\nabla)u\big|^2  \big\}\, dx  \nonumber\\ 
		& + \frac{N(N-2)}{N-1} \frac{1}{r^{N-1}}\int_{B_r}|u|^2\, dx + \frac{1}{N-1} \frac{1}{r^{N-3}} \int_{B_r} |\nabla u|^2 \, dx  + \frac{1}{r^{N-2}}\int_{\d B_r} x\cdot u \big(\frac{|u|^2}{2} + P \big)\, dS_x. \nonumber
	\end{align}
	Next, we also have that
	\begin{align*}
		&\frac{1}{r^{N-2}} \int_{\d B_r} x\cdot u \left(\frac{|u|^2}{2} + P\right) \, d S_x = \frac{1}{r^{N-3}} \int_{\d B_r} \frac{x}{|x|}\cdot u \left(\frac{|u|^2}{2} + P \right) \, d S_x\\ =& \frac{d}{d r} \bigg\{ \frac{1}{r^{N-3}} \int_{B_r} \frac{x}{|x|}\cdot u \left(\frac{|u|^2}{2} + P\right) d x \bigg\} + \frac{N-3}{r^{N-2}}\int_{B_r} \left(\frac{|u|^2}{2} + P\right)  u\cdot \frac{x}{|x|} \, d x.  
	\end{align*}
Plugging this into \eqref{temp3:dAdr} we obtain the desired result. 	
\end{proof}

\begin{lemma}\label{lem:positiveinf}
Let 
\begin{equation}
Q(r):= \int_{B_r} \Big\{ \frac{15}{4r^3}|u|^2 
+ 
\frac{1}{4r}|\nabla u |^2 
+  
\frac{3(r^2-|x|^2)}{4r^3} |\nabla u|^2\Big\}\, dx,
\end{equation}
and 
\begin{equation}
\text{\Large$\wp$}(r)  =\frac{1}{r^2} \int_{B_r}  \Big\{{|u|^2} + 2P \Big\} (u\cdot \frac{x}{|x|})  \, dx .
\end{equation}
If  $\liminf_{r\to 0} M(r)<\infty$ and 
\[
\liminf_{r\to0^+}\left[ Q(r)
+\text{\Large$\wp$}(r)
\right]> 0
\]
then $x=0$ is a regular point.
\end{lemma}

\begin{proof} 

Under the conditions  $A(r)$ is nondecreasing, and hence bounded 
since $|A(r)|\lesssim M(r)$ and there is 
a sequence $r_k$ such that $\lim_{k\to \infty} M(r_k)<\infty$. 
Applying Proposition \ref{prop:dAdr} we see that 
\[
A'(r)\ge \frac1r \int_{B_r} \frac{3}{4r^3}\left| \nabla (|x|u) \right|^2\ge 0.
\]
Moreover from Lemmas \ref{lem:Lin} and \ref{lem:Linbbb}
it follows that $\lim_{k\to \infty} M(r_k)\ge\delta>0$.
Introduce $u_k(x)=r_ku(r_k x)$, then for $0<\alpha<\beta$ we have 
\begin{align*}
0\leftarrow A[u](\beta r_k)-A[u](\alpha r_k)
&=
A[u_k](\beta)-A[u_k](\alpha )\\
&\ge 
\int_\alpha^\beta \frac1t \int_{B_t} \frac{3}{4t^3}\left| \nabla (|x|u_k) \right|^2\ge 0.
\end{align*}
Choosing a suitable subsequence $k_m$ and applying a customary 
compactness argument we can show that 
$u_{k_m}\to u_*$ weakly in $\mathcal W^{1,2}_{loc}(\R^5)$, $u_*$ is a suitable 
weak solution such that 
\[
\int_\alpha^\beta \frac1t \int_{B_t} \frac{3}{4t^3}\left| \nabla (|x|u_*) \right|^2=0
\]
hence $u_*$ is homogeneous function of 
degree negative one. Applying the result from \cite{Sverak} we conclude 
that $u_*=0$. On the other hand the condition 
$\lim_{k\to \infty} M(r_k)\ge\delta>0$ 
translates to $u_*$ and we conclude that $\int_{B_1}|u_{k_m}|^2+|\nabla u_{k_m}|^2>\delta/2.$
Due to strong convergence $u_{k_m}\to u_*$ we see that 
$\int_{B_1}|u_{k_m}|^2<\delta/4$. Thus 
$\int_{B_1}|\nabla u_{k_m}|^2>\delta/4$.
Using the local energy inequality \eqref{mu} we arrive at a contradiction.

\end{proof}

This lemma shows that 
there are three possibilities if $x=0$ is a singular point:
\begin{itemize}
\item[(1)] $\liminf_{r\to 0^+} M(r)=\infty$, 
\item[(2)] $\liminf_{r\to0^+}\left[ Q(r)+\text{\Large$\wp$}(r)\right]\le 0$,
\item[(3)] $\liminf_{r\to 0^+} M(r)=\infty$, 
and 
$\liminf_{r\to0^+}\left[ Q(r)+\text{\Large$\wp$}(r)\right]\le 0$.
\end{itemize}

Note that $Q(r)\sim M(r)$. 
Thus for the case $\liminf_{r\to 0^+} M(r)<\infty$
we need to analyze the behavior of $M(r)+\text{\Large$\wp$}(r)$.
The rest of the paper is devoted to this analysis. 

In fact, we will see that the condition in the Theorem \ref{TH} 
implies that 
$\text{\Large$\wp$}(r)$ is small compared to 
$\left(M(r)\right)^{3/2}$.

%
%

\section{Classification of self-similar solutions of degree -1 for the stationary Euler equations in $\R^5$} \label{sec:Euler}

\begin{theorem}\label{thm:Heuler}
	Suppose $V\in \mathcal W^{1, 2}_{loc}(\R^5), P\in L^1_{loc}(\R^5)$ have the form 
	\begin{align*}
		V=\frac{v(\sigma)+f(\sigma)\sigma}{|x|}, \quad P=\frac{p(\sigma)}{|x|^2}, \quad \sigma=\frac x{|x|}, 
	\end{align*}
	solves the Euler system 
	\begin{equation*}
		(V\cdot \nabla)V+\nabla P=0, \quad \div V=0, 
	\end{equation*}
	where $f, p, v$ are some function on $\Sph^4$. 
	Then $f=|v|^2+2p=0$.
\end{theorem}

\begin{proof}
We use a slightly general set up to emphasise the importance of dimension five. 
In $\R^N$ the Euler system in spherical coordinates takes the following form 
	\begin{align}\label{eq:EulerSphere}
		\left\{
		\begin{array}{ccc}
			(N-2)f+\div v=0\\
			v\cdot \nabla f=H\\
			v\cdot \nabla H=2fH
		\end{array}
		\right.
	\end{align}
	where $H=|v|^2+f^2+2p$. These equations are derived in Appendix \ref{sec:app-Euler}. 
	Note that the embedding theorem on $\mathbb S^{N-1}$ \cite{Beckner} implies that 
	$v, f\in L^{\frac{2(N-1)}{(N-1)-2}}(\Sph^{N-1})$. If $N=5$, then $v, f \in L^2(\mathbb S^4)$, which in turn 
	imply that $V\in L^4_{loc}(\R^5)$. Hence, applying the local estimates for the pressure Proposition \ref{prop:PE},   	we conclude that $P\in L^2_{loc}(\R^5)$ and $p\in L^2(\Sph ^4)$ and $H\in L^2(\Sph^4)$.
	
	Multiplying the second equation in \eqref{eq:EulerSphere} by $H$ and integrating by parts gives 
	\begin{align*}
		\int_{\mathbb S^{N-1}}H^2=\int_{\mathbb S^{N-1}} v\cdot \nabla f H=
		-\int_{\mathbb S^{N-1}}(\div vH+v\cdot \nabla H)f\\
		=-\int_{\mathbb S^{N-1}}(-(N-2)fH+2fH)f\\
		=(N-4)\int_{\mathbb S^{N-1}}Hf^2.
	\end{align*}
	Splitting $H^2=Hf^2 +H(|v|^2+2p)$ and rearranging the integrals yields
	\begin{align}\label{eq:Split1}
		\int_{\mathbb S^{N-1}}H(|v|^2+2p)=(N-5) \int_{\mathbb S^{N-1}}Hf^2.
	\end{align}

	Next, we multiply the second equation in \eqref{eq:EulerSphere} by $f^2$ and integrate  by parts
	
	\begin{align*}
		\int_{\mathbb S^{N-1}}Hf ^2=\int_{\mathbb S^n} v\cdot \nabla f f^2=
		-\int_{\mathbb S^{N-1}}\div v\frac{f^3}3\\
		=\frac{N-2}3\int_{\mathbb S^{N-1}}f^4.
	\end{align*}

	Splitting $f^2 H=f^4+ f^2(|v|^2+2p)$ and rearranging the integrals gives the following integral identity 
	
	\begin{align}\label{eq:Split2}
		\int_{\mathbb S^{N-1}}f^2(|v| ^2+2p)
		=\left(\frac{N-2}3-1\right)\int_{\mathbb S^{N-1}}f^4\\
		=\frac{N-5}3\int_{\mathbb S^{N-1}}f^4.
	\end{align}
	
	Taking $N=5$ in \eqref{eq:Split1} and  \eqref{eq:Split2} implies 
	\[
	\int_{\mathbb S^4}H(|v| ^2+2p)=\int_{\mathbb S^4}f^2(|v| ^2+2p)=0.
	\]
	Hence, subtracting the first integral from the last yields 
	\[
	\int_{\mathbb S^4}(|v| ^2+2p)^2=0.
	\]
	
	Thus $H=f^2$, hence from the first equation in \eqref{eq:EulerSphere} we get 
	\begin{align}
		(N-2)\int_{\mathbb S^{N-1}}f^2=-\int_{\mathbb S^{N-1}}\div vf=
		\int_{\mathbb S^{N-1}}v\cdot \nabla f=\int_{\mathbb S^{N-1}}H= \int_{\mathbb S^{N-1}}f^2,
	\end{align}
	implying  $(N-3)\int_{\mathbb S^{N-1}}f^2=0$. For  $N=5$ the result follows. 
\end{proof}

%
%

\section{Pressure Estimates}\label{sec:energy}

\begin{proposition}\label{prop:bogo}
	Suppose $u\in \mathcal W^{1,2}(B)$ and $p\in L^1_{\text{loc}}(B)$ solves the equations
	\begin{equation}\label{distrNS}
		\div (u\otimes u) + \nabla p = \Delta u \ \text{ in the sense of distribution in } \ x\in B.
	\end{equation}
	Then there exists a generic constant $C>0$ such that 
	\begin{equation*}
		 \big\| p - [p] \big\|_{ L^{5/3}(B) } \le C \| u \|_{\mathcal W^{1,2}(B)} \left\{  1 + \| u \|_{\mathcal W^{1,2}(B)} \right\}. 
	\end{equation*}
\end{proposition}
\begin{proof}
	We define $q\vcentcolon= p - [p]$. Then $\frac{1}{|B|}\int_{B} q = 0$ and it holds that $(u,q)$ also solves the equations \eqref{distrNS} in the sense of distributions. By Bogovski\v{i}'s theorem (See Section \textrm{III}.3 of \cite{Galdi}), there exists a vector-valued function $\psi \in \mathcal W_0^{1,5/2}(B)$ such that 
	\begin{equation}\label{psi}
		\left\{ 
		\begin{aligned}
			&\div \psi = g \vcentcolon= \sgn(q) |q|^{\frac{2}{3}} - \frac{1}{|B|}\int_{B} \sgn(q) |q|^{\frac{2}{3}} && \text{ in } \ x \in B_R,\\
			& \frac{x}{|x|}\cdot \psi(x)  = 0 && \text{ on } \ x \in \d B_R,			
		\end{aligned}
		\right.
	\end{equation}
	where $\sgn(q)$ is the sign function defined by $\sgn(q)= \frac{q}{|q|}$ if $q\neq 0$ and $\sgn(q)=0$ if $q=0$. In addition, it is shown in Section \textrm{III}.3 of \cite{Galdi} that there exists a generic constant $C>0$ such that 
	\begin{align}\label{psiIneq}
		\left\| \psi \right\|_{W^{1,5/2}(B)} \le C\left\| g  \right\|_{L^{5/2}(B)} \le C \left\| q \right\|_{L^{5/3}(B)}^{2/3}.
	\end{align}
	Replacing $p$ by $q$ in the second equation of distributional equalities \eqref{distrNS} then using $\psi\in \mathcal W_0^{1,5/2}(B)$ as a test function, we get
	\begin{align}\label{psiWF}
		\int_{B} q \div \psi = \int_{B} \left\{ u\otimes u \vcentcolon \nabla \psi + \nabla u \vcentcolon \nabla \psi \right\}.
	\end{align} 
	By construction \eqref{psi}, the left hand side of the above equation is given by
	\begin{equation}
		\int_{B} q \div \psi = \int_{B}|q|^{\frac{5}{3}} = \left\| q \right\|_{L^{5/3}(B)}^{5/3}.
	\end{equation}
	Substituting the above into \eqref{psiWF}, applying Poincar\'e's inequality, Sobolev's embedding theorem and \eqref{psiIneq}, we obtain
	\begin{align*}
		&\left\| q \right\|_{L^{5/3}(B)}^{5/3} = \int_{B} \left\{u\otimes u \vcentcolon \nabla \psi + \nabla u \vcentcolon \nabla \psi \right\}\\
		\le& C \left\{ \left\| u \right\|_{L^{10/3}(B)}^2 + \left\| \nabla u \right\|_{L^{2}} \right\} \left\| \nabla \psi \right\|_{L^{5/2}(B)} \le C \left\| u \right\|_{W^{1,2}(B)} \left\{ 1 + \left\| u \right\|_{\mathcal W^{1,2}(B)} \right\}  \left\| q \right\|_{L^{5/3}(B)}^{2/3}.
	\end{align*}  
	Dividing both sides by $\| q \|_{L^{5/3}(B)}^{2/3}$ gives the desired result.
\end{proof}

\begin{corollary}\label{corol:bogo-Scaled}
	Fix a radius $R>0$. Suppose $u\in \mathcal W^{1,2}(B_R)$, $p\in L^1_{\text{loc}}(B_R)$ is a Leray-Hopf weak solution to the Navier-Stokes equations. Then there exists a generic constant $C>0$ such that
	\begin{equation*}
		 \big\| p - [p]_R \big\|_{L^{5/3}(B_R)} \le 
		  C R \left\{ 1 + \sqrt{M(R)} \right\}\sqrt{M(R)},
	\end{equation*}
	where $M(R)\vcentcolon= M[u](R)$.
\end{corollary}
\begin{proof}
	For $(u,p)$ in $x\in B_R$, we define $(u^R,p^R)(y)\vcentcolon= \left(R u(Ry), R^2 p(Ry)\right)$ for $y\in B$. Then $(u^R,p^R)$ solves \eqref{distrNS} in the sense of distribution in $y\in B$. Applying Proposition \ref{prop:bogo} on $(u^R,p^R)$ then rescaling the domain of integral from $B$ to $B_R$ yield the desired inequality.
\end{proof}

\begin{proposition}[Pressure Estimate]\label{prop:PE}
	Fix a radius $R>0$. Suppose $v\in \mathcal W_{\text{loc}}^{1,2}(B_{R})$ and $q\in L_{\text{loc}}^1(B_{R})$ solve the equations
	\begin{equation}\label{vq}
		\div\,v =0, \qquad \Delta q = - \div \div (v\otimes v) \quad \text{ for } \ x\in B_{R} \ \text{ in the sense of distribution.}
	\end{equation}
	Then there exists a constant $C>0$ independent of $R>0$, $v$ and $q$ such that
	\begin{enumerate}[label=\textnormal{(\roman*)},ref=\textnormal{(\roman*)}]
		\item\label{item:PE1} if $v\in \mathcal W^{1,2}(B_{R})$ and $q\in L^1(B_{R})$ then
		\begin{gather*}
			\left\|q - \left[q\right]_R \right\|_{L^{5/3}(B_{R/2})} + \|\nabla q\|_{L^{5/4}(B_{R/2})} \le C \|v\|_{\mathcal  W^{1,2}(B_R)}\|\nabla v\|_{L^2(B_{R})} + \frac{C}{R^2}\|q - \left[q\right]_R \|_{L^1(B_R)},
		\end{gather*} 
		\item \label{item:PE2} if $v \in \mathcal  W^{1,2}(B_{R})\cap L^4(B_R)$ and $q\in L^1(B_R)$ then
		\begin{gather*}
			\| q - [q]_R \|_{L^2(B_R)}  \le C \| v \|_{L^4(B_R)}  + \frac{C}{R^2}\|q - \left[q\right]_R \|_{L^1(B_R)}.
		\end{gather*}
	\end{enumerate}
\end{proposition}
\begin{proof}
	
	Let $\varphi\in \mC_{c}^{\infty}(B_R)$ be a test function such that
	\begin{align*}
		0\le \varphi \le 1, \qquad \varphi =1 \ \text{ in } \ B_{R/2}, \qquad \varphi =0 \ \text{ in } \ B_R \backslash B_{3R/4}, \qquad |\nabla^k \varphi| \le \frac{C}{R^{|k|}} \ \text{ for } \ k\in\N^5, 
	\end{align*}
	where $C>0$ is some generic constant. Denote $\breve{q}(x)\vcentcolon= q(x) - [q]_R$. By equations \eqref{vq}, we have that
	\begin{align}\label{Deltaq}
		-\Delta \left(\varphi \breve{q}\right) =& \div \div \big(\varphi v \otimes (v-[v]_R) \big) - \div\big\{ (v-[v]_R) v\cdot \nabla \varphi + v (v-[v]_R)\cdot \nabla \varphi \big\}\\ &+ (v-[v]_R)\otimes v \vcentcolon \nabla^2 \varphi - \div\left(2\breve{q}\nabla \varphi\right) + \breve{q} \Delta\varphi\nonumber 
	\end{align}
	holds in $\R^5$ in the sense of distribution. Applying the operator $(-\Delta)^{-1}$ on both sides of the above equation, we have
	\begin{align}\label{phiq}
		\breve{q}(x)=\varphi(x) \breve{q}(x) 
		= q_1(x) +q_2(x)+q_3(x), 
	\end{align} 
	for $x\in B_{R/2}$ where 
	\begin{align*} 
		q_1\vcentcolon=&\sum_{i,j=1} \mR_i \mR_j \big( \varphi v^i (v-[v]_R)^j \big),\\
		q_2(x)\vcentcolon=& \!\int_{\R^5}\! \sum_{i=1}^5\frac{3(x_i-y_i)}{8\pi^2|y-x|^5}\big\{ (v^i-[v]_R^i) v\cdot \nabla \varphi + v^i (v-[v]_R)\cdot \nabla \varphi \big\} d y\\
		&+ \int_{\R^5}\! \frac{\nabla^2 \varphi}{8\pi^2|y-x|^3}\vcentcolon v\otimes(v-[v]_R)\, d y,\\
		q_3(x)\vcentcolon=& \int_{\R^5} \frac{3(x-y)}{4\pi^2|y-x|^5} \cdot \left(\breve{q} \nabla \varphi\right)(y)\, d y  + \int_{\R^5} \frac{(\breve{q} \Delta \varphi)(y)}{8\pi^2 |y-x|^3} dy.
	\end{align*}
	Here, $\mR=(\mR_1,\dotsc \mR_5)^\top$ is the Riesz transform and $\|\mR\|_{L^s\to L^s}\le C(s)<\infty$ for $1<s<\infty$. The $L^{5/3}(B_{R/2})$ norm of $q_1$ is estimated using Poincar\'e-Sobolev inequality as follows
	\begin{align*}
		\|q_1\|_{L^{5/3}(B_{R/2})} \le& C \big\| \varphi v\otimes (v-[v]_R) \big\|_{L^{5/3}(\R^5)} \le C \|v\|_{L^{10/3}(B_R)}  \|v-[v]_R\|_{L^{10/3}(B_R)}\\
		\le& C \|v\|_{W^{1,2}(B_R)}\|v-[v]_R\|_{W^{1,2}(B_R)} \le C \|v\|_{W^{1,2}(B_R)}\|\nabla v\|_{L^{2}(B_R)}.
	\end{align*}
	Since $\supp(\nabla \phi) \subseteq B_R\backslash B_{3R/4}$, it follows that if $x\in B_{R/2}$ and $y\in \supp(\nabla \phi)$, then $|x-y| \ge R/4$. Using this, we obtain the inequality that for $x\in B_{R/2}$, 
	\begin{align*}
		|q_2(x)| \le& C\int_{B_{R}\backslash B_{3R/4}}\!\!\Big(\frac{1}{R|x-y|^4} +\frac{1}{R^2|x-y|^3}\Big) |v|\cdot|v-[v]_R|\, d y\\ 
		\le& \frac{C}{R^3} \big\| |v|\cdot|v-[v]_R| \big\|_{L^{5/3}(B_R)} \le \frac{C}{R^3} \|v\|_{W^{1,2}(B_R)}\|\nabla v\|_{L^2(B_R)}.
	\end{align*} 
	Therefore, it follows that
	\begin{align*}
		\|q_2\|_{L^{5/3}(B_{R/2})}\le C \|v\|_{
		\mathcal W^{1,2}(B_R)}\|\nabla v\|_{L^2(B_R)}.
	\end{align*}
	For the term $q_3$, it follows similarly that if $x\in B_{R/2}$ then
	\begin{align*}
		|q_3(x)| \le C \int_{B_R\backslash B_{3R/4}} \Big( \frac{1}{R|x-y|^4} + \frac{1}{R^2|x-y|^3} \Big) |\breve{q}(y)| dy \le \frac{C}{R^5} \left\|q-[q]_R\right\|_{L^1(B_R\backslash B_{R/2})},
	\end{align*}
	which gives the following estimate
	\begin{equation*}
		\|q_3\|_{L^{5/3}(B_{R/2})} \le \frac{C}{R^2}\left\|q - \left[q\right]_R \right\|_{L^1(B_R\backslash B_{R/2})}.
	\end{equation*}
	Combining the estimates of $\{\|q_i\|_{L^{5/3}(B_{R/2})}\}_{i=1}^3$, we have
	\begin{equation*}
		\|\breve{q}\|_{L^{5/3}(B_{R/2})} \le C \|v\|_{\mathcal  W^{1,2}(B_R)}\|\nabla v\|_{L^2(B_R)} + \frac{C}{R^2}\left\|q-[q]_R\right\|_{L^1(B_R\backslash B_{R/2})}.
	\end{equation*} 
	This proves the $L^{5/3}$-estimate of $q-[q]_R$.
	
	Next, we derive the estimate for $\nabla q$. Taking derivative $\d_{x_k}$ on the equation \eqref{Deltaq} for $k=1,\dotsc,5$, then applying the operator $(-\Delta)^{-1}$, we get
	\begin{align*}
		\d_{x_k} q(x) = \d_{x_k} \breve{q} (x) = \tilde{q}_1^k(x) + \tilde{q}_2^k(x) + \tilde{q}_3^k(x),
	\end{align*}
	for $x\in B_{R/2}$ where
	\begin{align*}
		\tilde{q}_1^k \vcentcolon=& \sum_{i,j=1}^5 \mR_i\mR_j \big(\varphi (v-[v]_R)^j \nabla v^i +  \varphi v^i \nabla v^j  \big),\\
		\tilde{q}_2^k \vcentcolon=& \int_{\R^5}\sum_{i,j=1}^5 \frac{3(x_i-y_i)}{8\pi^2|y-x|^5}  \d_{y_k}\big\{ (v-[v]_R)^i v^j + v^i (v-[v]_R)^j \big\} \d_{y_j}\varphi\, dy \\
		&+ \int_{\R^5}\sum_{i,j=1}^5\frac{1}{8\pi^2|y-x|^3}\d_{y_k}\big\{ (v-[v]_R)^i v^j \big\} \d_{y_i} \d_{y_j} \varphi\, dy,\\ 
		\tilde{q}_3^k \vcentcolon=& \int_{\R^5} \sum_{i=1}^5 \frac{3 \left(\breve{q} \d_{y_i} \varphi\right)(y)}{4\pi^2|y-x|^5} \Big\{ \delta_{ik} - 5 \frac{(y_i-x_i)(y_k-x_k)}{|y-x|^2} \Big\}\, dy - \int_{\R^5} \frac{\left(\breve{q} \d_{y_k} \Delta \varphi\right)(y)}{8\pi^2|y-x|^3}\, dy\\
		&+ \int_{\R^5}\sum_{i=1}^5\frac{3(y_i-x_i)}{4\pi^2|y-x|^5}\left(\breve{q}\d_{y_k}\d_{y_i}\varphi\right)(y) \, dy + \int_{\R^5}\frac{3(y_k-x_k)}{8\pi^2 |y-x|^5}\left(\breve{q}\Delta\varphi\right)(y)\, dy. 
	\end{align*}
	By the interpolation inequality, $L^p$-boundedness of Riesz's operator and Poincar\'e-Sobolev inequality,
	\begin{align*}
		&\|\tilde{q}_1^k\|_{L^{5/4}(B_{R/2})}\\ \le& C \big\{ \| v-[v]_R \|_{L^{10/3}(B_R)} \|\nabla v\|_{L^2(B_R)} + \| v \|_{L^{10/3}(B_R)} \|\nabla v\|_{L^2(B_R)} \big\}\le C\|v\|_{\mathcal W^{1,2}(B_R)}\|\nabla v\|_{L^2(B_R)}.
	\end{align*} 
	The estimates of $\tilde{q}_2^k$ and $\tilde{q}_3^k$ are obtained using the same argument for the terms $q_2$ and $q_3$ defined above. That is, by the fact that $\supp(\nabla \varphi) \subseteq B_{R}\backslash B_{3R/4}$, the singular integral kernels in $\tilde{q}_2^k$ and $\tilde{q}_3^k$ are bounded for $x\in B_{R/2}$. Thus, it follows that 
	\begin{align*}
		\| \tilde{q}_2^k \|_{L^{5/4}(B_{R/2})} \le C \|v\|_{\mathcal  W^{1,2}(B_R)}\|\nabla v\|_{L^2(B_R)}, \qquad \|\tilde{q}_3^k\|_{L^{5/4}(B_{R/2})} \le \frac{C}{R^2} \left\|q-[q]_R\right\|_{L^1(B_R)}. 
	\end{align*} 
	This shows the $L^{5/4}$-estimate for $\nabla p$ hence completes the proof of \ref{item:PE1}. Finally if $v\in L^4$ then taking $L^2(B_{R/2})$-norm on both sides of equation \eqref{phiq} then repeating the same argument as before, we also obtain \ref{item:PE2}.
\end{proof}

\begin{proposition}[Homogeneity of Pressure]\label{prop:HP}
	Suppose $h\in \mH(1)$ and $p\in L^1(B)$ solve the equations
	\begin{align*}
		\div h =0 \quad \text{ and } \quad (h\cdot \nabla) h + \nabla p = 0 \quad \text{for } \ x\in B \ \text{ in the sense of distribution.}
	\end{align*}
	Then there exists a constant $p_0\in\R$ and $\xi\in \mathcal W^{1,5/4}(\Sph^4)$ such that 
	\begin{align*}
		p = \frac{1}{|x|^2} \xi\big(\frac{x}{|x|}\big) + p_0 \qquad \text{for } \ x\in B \ \text{ a.e.}
	\end{align*}
\end{proposition}
\begin{proof}
	First, under the assumption of the proposition, it holds that
	\begin{align*}
		-\Delta p = \div \div (h\otimes h) \qquad \text{for } \ x \in B \ \text{ in the sense of distributions.}
	\end{align*}
	By Proposition \ref{prop:PE}, $p\in L^{5/3}(B_{1/2})$ and $\nabla p \in L^{5/4}(B_{1/2})$. 
	We denote $\sigma_i\equiv \frac{x_i}{|x|}$ for $i=1,\dotsc,5$. Let $\zeta\in \mathcal W^{1,2}(\Sph^4)$ be such that $h(x)=\frac{1}{|x|}\zeta(\sigma)$. Then it can be verified that for $x\in B$ a.e.
	\begin{equation}\label{nablah}
		\d_{j} h^i = \frac{1}{|x|^2}\big\{ \big(\nabla_{\Sph^4} \zeta^i\big)^{j} - \sigma_j\zeta^i \big\}, 
	\end{equation} 
	where $\nabla_{\Sph^4}$ is the derivative on the sphere given by
	\begin{equation*}
		\big(\nabla_{\Sph^4} f\big)^{j} = \sum_{k=1}^5 |x|\Big\{ \delta_{jk} - \frac{x_j x_k}{|x|^2} \Big\} \frac{\d f}{\d x_k},
	\end{equation*}
	for continuously differentiable scalar functions $f(x) \vcentcolon \R^5 \to \R $. By \eqref{nablah}, we have that the following equation holds for $x\in B_{1/2}$ a.e.
	\begin{equation} \label{nablap}
		\nabla p = (h\cdot \nabla) h^i 
		= \frac{\omega^i(\sigma)}{|x|^3} \qquad \text{where } \ \omega^i(\sigma) \vcentcolon= \sum_{j=1}^5 \big(\nabla_{\Sph^4} \zeta^i \big)_{j} \zeta^j - (\sigma \cdot \zeta) \zeta^i.
	\end{equation}
	Now fix $\sigma_0 \in \Sph^4$. For $x\in B_{R/2}$, we set $\sigma = \frac{x}{|x|} \in \Sph^4$. Then there exists a mapping $\gamma(t)\vcentcolon [0,1]\to \Sph^4$ such that $\gamma \in \mC^{\infty}([0,1])$, $\gamma(0)=\sigma_0$ and $\gamma(1)=\sigma$. Note that $\dot{\gamma}\cdot \gamma = 0$. Taking inner-product of $|x|\dot{\gamma}(t)$ and equation \eqref{nablap} then integrating in $t\in [0,1]$, we get
	\begin{align*}
		p(x) - p(|x|\sigma_0) = \frac{1}{|x|^2}\int_0^1  \dot{\gamma} \cdot \omega\big(\gamma(t)\big)\, dt 
	\end{align*} 
	
	From here we can write 
	\[
	p(x)=D(|x|)+\frac{\xi(\sigma)}{|x|^2}
	\]
	Since $p$ has weak derivatives then it follows that $D$ is differentiable. Consequently, 
	\[
	\partial _r p=D'(r)-2\frac{\xi(\sigma)}{r^3}
	\]
	From \eqref{nablap} we have that the radial derivative of $p$ is 
	\[
	\partial_r p=\frac{\omega(\sigma)\cdot \sigma}{r^3}. 
	\]
	This yields
	\[
	r^3 D'(r)=2\xi(\sigma)+\omega(\sigma)\cdot \sigma,
	\]
	forcing both sides to be constant. Solving this equation we get $D(r)=\frac{C_1}{r^2}+C_2$, 
	where $C_1, C_2$ are constants. Summarizing, we see that 
	\[
	p(x)=C_2+\frac{\xi(\sigma)+C_1}{|x|^2}.
	\]
\end{proof}

%
%
\section{Main proposition}\label{sec:technical}
Before proving our main propositions, we introduce the following function space.
\begin{definition}
	Fix $R>0$. $\mH(R)$ is the subspace of $\mathcal W^{1,2}(B_R)$ defined by
	\begin{equation*}
		\mH(R) \vcentcolon= \Big\{ h\in \mathcal  W^{1,2}(B_R) \, \Big\vert \, \exists \zeta \in \mathcal W^{1,2}(\Sph^{N-1}): \ h(x)= \frac{1}{|x|} \zeta\big(\frac{x}{|x|}\big) \ \text{ for } \ x\in B_R \ \text{ a.e.}  \Big\}.
	\end{equation*}
\end{definition}
It can be verified that $\mH(R)$ is a closed linear subspace of $\mathcal W^{1,2}(B_R)$ with the scaled Sobolev norm
\begin{equation*}
	\|u\|_{\mathcal{W}^{1,2}_R}\vcentcolon=\bigg(\int_{B_R} \frac{1}{R^{N-2}}|u|^2 + \frac{1}{R^{N-4}}|\nabla u|^2\bigg)^{1/2}.
\end{equation*}
For $h\in \mH(R)$, let $\zeta\in \mathcal  W^{1,2}(\Sph^{N-1})$ be such that $h(x)=\frac{1}{|x|}\zeta\big(\frac{x}{|x|}\big)$ for a.e. $x\in B_R$. Then with few lines of calculation, one verifies that for $N\ge 5$
\begin{equation}\label{W12}
	\|h\|_{L^2(B_R)}^2 = \frac{R^{N-2}}{N-2} \|\zeta\|_{L^2(\Sph^{N-1})}^2, \qquad  \|\nabla h\|_{L^{2}(B_R)}^2= \frac{R^{N-4}}{N-4} \|\zeta\|_{\mathcal  W^{1,2}(\Sph^{N-1})}^2.
\end{equation}
Moreover, $\mathcal W^{1,2}(B_R)$ endowed with the inner-product
\begin{equation*}
	\langle u,v\rangle_{\mathcal{W}^{1,2}_R} \vcentcolon= \int_{ B_R}\frac{1}{R^{N-2}}u\cdot v+\frac{1}{R^{N-4}}\nabla u\vcentcolon\nabla v
\end{equation*}
is also a Hilbert space. It follows by Hilbert projection theorem that $\mathcal  W^{1,2}(B_R) = \mH(R) \oplus \mH(R)^{\perp}$. Using this, we define $\mP_R [u] \!\in\! \mH(R)$ for $u\!\in\! \mathcal W^{1,2}(B_R)$ to be the unique function such that
\begin{equation}
	\big\|\mP_R [u] - h \big\|_{\mathcal  W^{1,2}_R} = \min\limits_{h\in \mH(R)} \| u-h \|_{W^{1,2}_R}.
\end{equation}

\begin{proposition}\label{prop:cubicEst}
	For fixed constants $\delta_{1}, \delta_{2}>0$, there exists $\ep>0$ such that if $(u,p)$ is a Leray-Hopf solution to the Navier-Stokes equations satisfying
	\begin{equation}\label{h-A}
		\frac{1}{R^3} \int_{B_R} \big|u- \mP_R[u] \big|^2 + \frac{1}{R} \int_{ B_R} \big|\nabla u - \nabla \mP_R[u] \big|^2 \le \ep M[u](R)
	\end{equation} 
	for some $R>0$, then 
	\begin{equation}
		\bigg|\frac{1}{R^2}\int_{B_{R/2}} \big( |u|^2 + 2 p \big) \, u \cdot \frac{x}{|x|} \bigg|  \le \delta_{1} + \delta_{2} \big(M[u](R)\big)^{\frac{3}{2}}. 
	\end{equation}
\end{proposition}
\begin{proof}
	We prove by contradiction. Suppose otherwise that for fixed $\delta_{1}>0$ and $\delta_{2}>0$, there exists a sequence of solutions $\{ (u^k, p^k) \}_{k\in\N}$ and sequences of strictly positive numbers $\{\ep_k\}_{k\in\N}$ and $\{ R_k \}_{k\in \N}$ such that 
	\begin{subequations}
		\begin{gather}
			\ep_k \to 0 
			\qquad \text{ as } \ k\to\infty,\label{CA1}\\
			\frac{1}{R_k^3}\int_{B_{R_k}} \big\lvert u^k-h^k\big\rvert^2 + \frac{1}{R_k} \int_{B_{R_k}} \big\lvert \nabla u^k - \nabla h^k \big\rvert^2 < \ep_k M_k,\label{CA2}\\
			\bigg\lvert \frac{1}{R_k^2} \int_{B_{R_k/2}} \big( |u^k|^2 + 2p^k \big) \, u^k\cdot \frac{x}{|x|} \bigg\rvert \ge \delta_{1} + \delta_{2} M_k^{\frac{3}{2}},\label{CA3} 
		\end{gather}
	\end{subequations}
	where the positive number $M_k$ and function $h^k$ are defined by
	\begin{equation}
		M_k \vcentcolon= \frac{1}{R_k^3} \int_{B_{R_k}} |u^k|^2 + \frac{1}{R_k} \int_{B_{R_k}} |\nabla u^k|^2, \qquad h^k \vcentcolon= \mP_{R_k} [u^k].
	\end{equation}
	From here, we divide the proof into 2 cases:
	\begin{align}
		 \label{case-1} \tag*{\textbf{Case 1}} & \limsup_{k\to \infty} M_k <\infty,\\
		 \label{case-2} \tag*{\textbf{Case 2}} & \limsup_{k\to \infty} M_k = \infty.
	\end{align}
	
	\paragraph{\textbf{Case 1:}} $\limsup_{k\to \infty} M_k <\infty$. For each $k\in\N$ and $y\in B$, we define the functions
	\begin{gather*}
		\bar{u}^k(y) \vcentcolon= R_k u^k(R_k y), \qquad \bar{h}^k(y) \vcentcolon= R_k h^k(R_k y),\\
		\hat{p}^k(y) \vcentcolon= R_k^2 p^k(R_k y), \qquad \bar{p}^k(y) \vcentcolon= \hat{p}^k(y) - \left[ \hat{p}^k \right].
	\end{gather*}
	Then for each $k\in\N$, $\big(\bar{u}^k,\bar{p}^k\big)(y)$ solves the equations
	\begin{equation}\label{NSBL}
		\div \bar{u}^k = 0, \qquad (\bar{u}^k\cdot \nabla)\bar{u}^k + \nabla \bar{p}^k = \Delta \bar{u}^k, \qquad -\Delta \bar{p}^k = \div \div \big( \bar{u}^k \otimes \bar{u}^k \big),
	\end{equation}
	for $y\in B$ in the sense of distribution. In addition, set $M=\sup_{k\in \N} M_k<\infty$. Then inequalities \eqref{CA2}--\eqref{CA3} and the condition $\div \bar{u}^k = 0$ yield
	\begin{subequations}
		\begin{gather}
			\| \bar{u}^k \|_{\mathcal W^{1,2}(B)}^2 \le M, \qquad \int_{B} \big\{ |\bar{u}^k - \bar{h}^k|^2 + |\nabla \bar{u}^k-\nabla \bar{h}^k|^2 \big\} \le M\ep_k,\label{NSCA1}\\
			\bigg\lvert \int_{B_{1/2}} \big( |\bar{u}^k|^2 + 2 \bar{p}^k \big)\,\bar{u}^k \cdot \frac{y}{|y|} \, d y \bigg\rvert = \bigg\lvert \frac{1}{R_k^2} \int_{B_{R_k/2}} \big( |u^k|^2 + 2 p^k \big)\,  u^k\cdot \frac{x}{|x|} \, d x \bigg\rvert 
			\ge \delta_{1}.\label{NSCA2}
		\end{gather}
	\end{subequations} 
	By \eqref{NSCA1} and Sobolev embedding theorem, the sequence $\{\bar{h}_k\}_{k\in\mathbb{N}}$ satisfies
	\begin{align}\label{barhEst}
		\sup\limits_{k\in\N} \big\| \bar{h}^k \big\|_{\mathcal  \mathcal W^{1,2}(B)} \le& \sup\limits_{k\in\N}  \big\| \bar{u}^k - \bar{h}^k \big\|_{\mathcal W^{1,2}(B)} + \sup\limits_{k\in\N} \big\| \bar{u}^k \big\|_{\mathcal W^{1,2}(B)} \le 2\sqrt{M}.
	\end{align}
	Since $h^k\in \mH(R_k)$ for each $k\in\N$, there exists $\zeta^k \in \mathcal W^{1,2}(\Sph^4) $ such that $h^k(x) = \frac{1}{|x|}\zeta^k\big(\frac{x}{|x|}\big)$ for a.e. $x\in B_{R_k}$. Thus $\bar{h}^k(y) = \frac{1}{|y|}\zeta^k\big(\frac{y}{|y|}\big)$ for a.e. $y\in B$. Furthermore by \eqref{W12}, we have the uniform estimate
	\begin{equation*}
		\sup\limits_{k\in\N} \big\lVert \zeta^k \big\rVert_{\mathcal  W^{1,2}(\Sph^4)}^2 \le 3 \sup\limits_{k\in\N} \big\lVert \bar{h}^k \big\rVert_{\mathcal W^{1,2}(B)} \le 6\sqrt{M}.
	\end{equation*}  
	By Rellich-Kondrachov compactness theorem, there exists $\bar{\zeta}^{\infty}\in \mathcal W^{1,2}(\Sph^4)$ and a subsequence in $\mathcal W^{1,2}(\Sph^4)$, which is still denoted as $\{\zeta^{k}\}_{k\in\N}$ for simplicity, such that as $k\to \infty$,
	\begin{align*}
		\zeta^{k} \to \bar{\zeta}^{\infty} \ \text{ strongly in } \ L^2(\Sph^4) \quad \text{and} \quad \nabla_{\Sph^4} \zeta^k \rightharpoonup \nabla_{\Sph^4} \bar{\zeta}^{\infty} \ \text{ weakly in } \ L^2(\Sph^4),
	\end{align*}
	where $\nabla_{\Sph^4}$ denotes the derivative on the sphere $\Sph^4$. Define
	\begin{equation}\label{hinfty}
		\bar{h}^{\infty}(y) \vcentcolon= \frac{1}{|y|} \bar{\zeta}^{\infty}\big( \frac{y}{|y|} \big).
	\end{equation}
	It can be verified that $\bar{h}^{\infty}\in \mathcal W^{1,2}(B)$. Since $\bar{h}^k(y) = \frac{1}{|y|}\zeta^k\big(\frac{y}{|y|}\big)$ for a.e. $y\in B$, it follows that 
	\begin{equation*}
		\bar{h}^{k} \to \bar{h}^{\infty} \ \text{ strongly in } \ L^2(B) \quad \text{ and } \quad \nabla \bar{h}^k \rightharpoonup \nabla \bar{h}^{\infty} \ \text{ weakly in } \ L^2(B)
		\ \text{ as } k\to\infty.
	\end{equation*}
	Combining the above with \eqref{CA1} and \eqref{NSCA1} provides the following convergences
	\begin{equation}\label{NScvg1}
		\bar{u}^k \to \bar{h}^{\infty} \ \text{ strongly in } \ L^2(B) \quad \text{and} \quad  \nabla \bar{u}^k \rightharpoonup \nabla \bar{h}^{\infty} \ \text{ weakly in } \ L^2(B) \quad  \text{ as } \ k \to \infty.
	\end{equation}
	By Sobolev embedding theorem, we have $\sup_{k\in\N}\| \bar{u}^k \|_{L^{10/3}(B)} \le C \sup_{k\in\N}\| \bar{u}^k \|_{\mathcal  W^{1,2}(B)} \le C \sqrt{M} $. Moreover by interpolation inequality, Sobolev embedding theorem and \eqref{NScvg1}, it follows that
	\begin{equation*}
		\left\| \bar{u}^k - \bar{h}^{\infty} \right\|_{L^3} \le \left\| \bar{u}^k - \bar{h}^{\infty} \right\|_{L^2}^{1/6} \left\| \bar{u}^k - \bar{h}^{\infty}  \right\|_{L^{10/3}}^{5/6} \le C \left\| \bar{u}^k - \bar{h}^{\infty} \right\|_{L^2}^{1/6} \left\{  \left\| \bar{u}^k \right\|_{\mathcal W^{1,2}} +  \left\| \bar{h}^{\infty}  \right\|_{\mathcal W^{1,2}}  \right\}^{5/6} \to 0,
	\end{equation*}
	as $k\to \infty$. Thus there exists a further subsequence, which is still denoted as $\{\bar{u}^k\}_{k\in\N}$ such that
	\begin{equation}\label{NScvg2}
		\bar{u}^k \to \bar{h}^{\infty} \ \text{ strongly in } \ L^3(B) \quad \text{and} \quad  \bar{u}^k \rightharpoonup \bar{h}^{\infty} \ \text{ weakly in } \ L^{10/3}(B) \quad  \text{ as } \ k \to \infty.
	\end{equation} 
	Next, we wish to obtain convergences for the pressure sequence $\{\bar{p}^k\}_{k\in\N}$. By construction, the pair $\left(\bar{u}^k,\hat{p}^k\right)$ solves the equations \eqref{distrNS} and \eqref{vq} in $y\in B$. Thus we apply Propositions \ref{prop:bogo}, \ref{prop:PE} and \eqref{NSCA1} to obtain the uniform estimate
	\begin{align*}
		&\sup\limits_{k\in\N}\Big\{ \left\|\bar{p}^k\right\|_{L^{5/3}(B_{1/2})} + \left\|\nabla \bar{p}^k\right\|_{L^{5/4}(B_{1/2})} \Big\} = \sup\limits_{k\in\N}\left\{ \left\|\hat{p}^k - \left[\hat{p}^k\right] \right\|_{L^{5/3}(B_{1/2})} + \big\|\nabla \bar{p}^k\big\|_{L^{5/4}(B_{1/2})} \right\}\\ 
		\le& C \sup\limits_{k\in\N}\big\| \bar{u}^k \big\|_{\mathcal  W^{1,2}(B)} \big\|\nabla \bar{u}^k \big\|_{L^2(B)} + C \sup\limits_{k\in\N}\left\| \hat{p}^k - \left[\hat{p}^k\right] \right\|_{L^1(B)}\\
		\le& C \| \bar{u}^k \|_{\mathcal W^{1,2}(B)} \left\{  1 + \left\| \bar{u}^k \right\|_{\mathcal  W^{1,2}(B)} \right\} \le C \sqrt{M}\left(1+ \sqrt{M}\right),
	\end{align*} 
	where $C>0$ is some generic constant. By Rellich-Kondrachov compactness theorem, there exist a function $\bar{p}^{\infty}\in L^{5/3}(B_{1/2})\cap \mathcal W^{1,5/4}(B_{1/2})$ and a subsequence $\{ \bar{p}^k \}_{k\in\N}$ such that as $k\to \infty$,
	\begin{subequations}\label{NScvg3}
		\begin{gather}
			\bar{p}^k \to \bar{p}^{\infty} \ \text{ strongly in } \ L^{5/4}(B_{1/2}),\\ \nabla \bar{p}^k \rightharpoonup \nabla \bar{p}^{\infty} \ \text{ weakly in } L^{5/4}(B_{1/2}), \qquad 
			\bar{p}^k \rightharpoonup \bar{p}^{\infty} \ \text{ weakly in } L^{5/3}(B_{1/2}).
		\end{gather}
	\end{subequations}
	Applying the convergences \eqref{NScvg1}--\eqref{NScvg3} on the equations \eqref{NSBL}, we have $(\bar{h}^{\infty},\bar{p}^{\infty})$ solves
	\begin{equation}
		\div \bar{h}^{\infty} = 0 \quad \text{ and } \quad \big(\bar{h}^{\infty}\cdot\nabla \big) \bar{h}^{\infty} + \nabla \bar{p}^{\infty} = \Delta \bar{h}^{\infty}
	\end{equation}
	for $y\in B_{1/2}$ in the sense of distribution. Moreover, applying convergences \eqref{NScvg1}--\eqref{NScvg3} to the inequality \eqref{NSCA2} yields 
	\begin{equation}\label{NSCA3}
		\bigg| \int_{B_{1/2}} \big( \big|\bar{h}^{\infty}\big|^2 + 2 \bar{p}^{\infty} \big) \bar{h}^{\infty} \cdot \frac{y}{|y|} \bigg| \ge \delta_{1}.
	\end{equation}
	By the classfication theorem of the homogeneous solution of degree $-1$ to the Navier-Stokes equations \cite{Sverak}, it follows that $\bar{h}^{\infty}=0$, $\bar{p}^{\infty}=p_0$ for some constant $p_0\in \R$. This is a contradiction to the inequality \eqref{NSCA3}.  
	
	\paragraph{\textbf{Case 2:}} $\limsup_{k\to\infty} M_k= \infty$. For each $k\in\N$ and $y\in B$, we define the functions
	\begin{gather*}
		\tilde{u}^k(y) \vcentcolon = \frac{R_k u^k(R_k y)}{\sqrt{M_k}}, \qquad \tilde{h}^k(y) \vcentcolon= \frac{R_k h^k(R_k y)}{\sqrt{M_k}},\\
		\hat{p}^k(y)\vcentcolon= R_k^2 p^k(R_k y), \qquad \tilde{p}^k(y) \vcentcolon = \frac{1}{M_k} \left(\hat{p}^k(y) - \left[\hat{p}^k\right]\right).
	\end{gather*}
	Then for each $k\in\N$, $\left(\tilde{u}^k,\tilde{p}^k\right)(y)$ solves the equations
	\begin{equation}\label{eulerBL}
		\div\tilde{u}^k =0, \qquad (\tilde{u}^k\cdot \nabla)\tilde{u}^k + \nabla \tilde{p}^k = \frac{1}{\sqrt{M_k}}\Delta \tilde{u}^k, \qquad -\Delta \tilde{p}^k = \div \div \big( \tilde{u}^k \otimes \tilde{u}^k \big),  
	\end{equation}
	for $y\in B$ in the sense of distribution. Moreover, by definitions of $\tilde{u}^k$ and $M_k$ we have 
	\begin{align}\label{tildeu1}
		\left\| \tilde{u}^k \right\|_{\mathcal W^{1,2}(B)}^2 = \frac{1}{M_k} \int_{B_{R_k}} \left\{ \frac{\left|u^k(x)\right|^2}{R_k^3} + \frac{\left|\nabla u^k(x)\right|}{R_k}  \right\} \, dx = 1,
	\end{align}
	for all $k\in\mathbb{N}$. Since $\left(\sqrt{M_k}\tilde{u}^k , \hat{p}^k\right)$ is a solution to \eqref{distrNS} in $y\in B$, we apply Proposition \ref{prop:bogo} to obtain
	\begin{equation*}
		\left\| \hat{p}^k - \left[\hat{p}^k\right] \right\|_{L^{5/3}(B)} 
		\le C \sqrt{M_k} \left\|  \tilde{u}^k \right\|_{\mathcal W^{1,2}(B)} \left\{1+ \sqrt{M_k} \left\|  \tilde{u}^k \right\|_{\mathcal W^{1,2}(B)} \right\}.
	\end{equation*}
	Dividing both sides by $M_k$, it follows by \eqref{tildeu1} that
	\begin{align*}
		\sup\limits_{k\in\mathbb{N}}\left\| \tilde{p}^k \right\|_{L^{5/3}(B)}=\sup\limits_{k\in\mathbb{N}}\left\| \frac{\hat{p}^k - \left[\hat{p}^k\right]}{M_k} \right\|_{L^{5/3}(B)} \le C \sup\limits_{k\in\mathbb{N}} \left\|  \tilde{u}^k \right\|_{\mathcal W^{1,2}(B)} \left\{ \frac{1}{\sqrt{M_k}}+ \left\|  \tilde{u}^k \right\|_{\mathcal W^{1,2}(B)} \right\}\le C. 
	\end{align*}
	Therefore by the above estimates, inequalities \eqref{CA2}--\eqref{CA3} and divergence free condition $\div \tilde{u}^k =0$, there exists a generic constant $C>0$ such that for all $k\in\N$
	\begin{subequations}
		\begin{gather}
			\|\tilde{u}^k\|_{\mathcal  W^{1,2}(B)} + \|\tilde{p}^k\|_{L^{5/3}(B)} \le C, \qquad \int_{B} \big\{ |\tilde{u}^k - \tilde{h}^k|^2 + |\nabla \tilde{u}^k-\nabla \tilde{h}^k|^2 \big\} \le \ep_k,\label{EulerCA1}\\
			\bigg\lvert \int_{B_{1/2}} \big( |\tilde{u}^k|^2 + 2 \tilde{p}^k \big)\,\tilde{u}^k \cdot \frac{y}{|y|}\, dy \bigg\rvert  
			= \left| \frac{1}{M_k^{3/2}R_k^2}\int_{B_{R_k/2}} \left( \left|u^k\right|^2 + 2 p^k \right)\, u^k \cdot \frac{x}{|x|} \, dx  \right|\ge \delta_{2}.\label{EulerCA2}
		\end{gather}
	\end{subequations}
	By \eqref{EulerCA1} and Sobolev embedding theorem, the sequence $\{\tilde{h}_k\}_{k\in\mathbb{N}}$ satisfies
	\begin{align}\label{hkEst}
		\sup\limits_{k\in\N} \big\| \tilde{h}^k \big\|_{\mathcal  W^{1,2}(B)} \le& \sup\limits_{k\in\N}  \big\| \tilde{u}^k - \tilde{h}^k \big\|_{W^{1,2}(B)} + \sup\limits_{k\in\N} \big\| \tilde{u}^k \big\|_{\mathcal  W^{1,2}(B)} 
		\le \sup\limits_{k\in\N} \sqrt{\ep_k} + 1 = 2. 
	\end{align}
	Since $h^k\in \mH(R_k)$ for each $k\in\N$, there exists $\zeta^k \in \mathcal W^{1,2}(\Sph^4) $ such that $h^k(x) = \frac{1}{|x|}\zeta^k\big(\frac{x}{|x|}\big)$ for a.e. $x\in B_{R_k}$. If we define $\tilde{\zeta}^k \vcentcolon= M_k^{-1/2}\zeta^k$, then $\tilde{h}^k(y) = \frac{1}{|y|}\tilde{\zeta}^k\big(\frac{y}{|y|}\big)$ for a.e. $y\in B$. Furthermore by \eqref{W12}, we have the uniform estimate
	\begin{equation}
		\sup\limits_{k\in\N} \big\lVert \tilde{\zeta}^k \big\rVert_{\mathcal  W^{1,2}(\Sph^4)}^2 \le 3 \sup\limits_{k\in\N} \big\lVert \tilde{h}^k \big\rVert_{\mathcal  W^{1,2}(B)} \le 6.
	\end{equation}  
	By Rellich-Kondrachov compactness theorem, there exists $\tilde{\zeta}^{\infty}\in \mathcal W^{1,2}(\Sph^4)$ and a subsequence in $\mathcal W^{1,2}(\Sph^4)$, which is still denoted as $\{\tilde{\zeta}^{k}\}_{k\in\N}$ for simplicity, such that as $k\to \infty$,
	\begin{align*}
		\tilde{\zeta}^{k} \to \tilde{\zeta}^{\infty} \ \text{ strongly in } \ L^2(\Sph^4) \quad \text{and} \quad \nabla_{\Sph^4} \tilde{\zeta}^k \rightharpoonup \nabla_{\Sph^4} \zeta^{\infty} \ \text{ weakly in } \ L^2(\Sph^4),
	\end{align*}
	where $\nabla_{\Sph}$ denotes the derivative on the sphere $\Sph^4$. Define
	\begin{equation*}
		\tilde{h}^{\infty}(y) \vcentcolon= \frac{1}{|y|} \tilde{\zeta}^{\infty}\big( \frac{y}{|y|} \big).
	\end{equation*}
	It can be verified that $\tilde{h}^{\infty}\in \mathcal  W^{1,2}(B)$. Since $\tilde{h}^k(y) = \frac{1}{|y|}\tilde{\zeta}^k\big(\frac{y}{|y|}\big)$ for a.e. $y\in B$, it follows that 
	\begin{equation*}
		\tilde{h}^{k} \to \tilde{h}^{\infty} \ \text{ strongly in } \ L^2(B) \quad \text{ and } \quad \nabla \tilde{h}^k \rightharpoonup \nabla \tilde{h}^{\infty} \ \text{ weakly in } \ L^2(B)
		\ \text{ as } k\to\infty.
	\end{equation*}
	Combining the above with \eqref{CA1} and \eqref{EulerCA1} provides the following convergences
	\begin{equation}\label{cvg1}
		\tilde{u}^k \to \tilde{h}^{\infty} \ \text{ strongly in } \ L^2(B) \quad \text{and} \quad  \nabla \tilde{u}^k \rightharpoonup \nabla \tilde{h}^{\infty} \ \text{ weakly in } \ L^2(B) \quad  \text{ as } \ k \to \infty.
	\end{equation}
	By Sobolev embedding theorem, we have $\sup_{k\in\N}\| \tilde{u}^k \|_{L^{10/3}(B)} \le C \sup_{k\in\N}\| \tilde{u}^k \|_{\mathcal  W^{1,2}(B)} \le C $. By interpolation inequality, Sobolev embedding theorem and \eqref{cvg1}, there exists a further subsequence, which is still denoted as $\{\tilde{u}^k\}_{k\in\N}$ such that
	\begin{equation}\label{cvg2}
		\tilde{u}^k \to \tilde{h}^{\infty} \ \text{ strongly in } \ L^3(B) \quad \text{and} \quad  \tilde{u}^k \rightharpoonup \tilde{h}^{\infty} \ \text{ weakly in } \ L^{10/3}(B) \quad  \text{ as } \ k \to \infty.
	\end{equation} 
	Next, we wish to obtain convergences for the pressure sequence $\{\tilde{p}^k\}_{k\in\N}$. Since $\big(\tilde{u}^k,\tilde{p}^k\big)$ solves the first and third equations of \eqref{eulerBL}, we can apply Proposition \ref{prop:PE} and \eqref{EulerCA1} to obtain the uniform estimate
	\begin{align*}
		&\sup\limits_{k\in\N}\left\|\nabla \tilde{p}^k\right\|_{L^{5/4}(B_{1/2})}\le C \sup\limits_{k\in\N}\big\| \tilde{u}^k \big\|_{\mathcal  W^{1,2}(B)} \big\|\nabla \tilde{u}^k \big\|_{L^2(B)} + C \sup\limits_{k\in\N}\| \tilde{p}^k \|_{L^1(B)}\le C,
	\end{align*} 
	where $C>0$ is some generic constant. By Rellich-Kondrachov compactness theorem, there exist a function $\tilde{p}^{\infty}\in L^{5/3}(B_{1/2})\cap \mathcal  W^{1,5/4}(B_{1/2})$ and a subsequence $\{ \tilde{p}^k \}_{k\in\N}$ such that as $k\to \infty$,
	\begin{subequations}\label{cvg3}
		\begin{gather}
			\tilde{p}^k \to \tilde{p}^{\infty} \ \text{ strongly in } \ L^{5/4}(B_{1/2}),\\ \nabla \tilde{p}^k \rightharpoonup \nabla \tilde{p}^{\infty} \ \text{ weakly in } L^{5/4}(B_{1/2}), \qquad 
			\tilde{p}^k \rightharpoonup \tilde{p}^{\infty} \ \text{ weakly in } L^{5/3}(B).
		\end{gather}
	\end{subequations}
	Applying the convergences \eqref{cvg1}--\eqref{cvg3} on the equations \eqref{eulerBL} and using the fact that $M_k\to \infty$, we have $(\tilde{h}^{\infty},\tilde{p}^{\infty})$ solves
	\begin{equation}
		\div \tilde{h}^{\infty} = 0 \quad \text{ and } \quad \big(\tilde{h}^{\infty}\cdot\nabla \big) \tilde{h}^{\infty} + \nabla \tilde{p}^{\infty} = 0
	\end{equation}
	in the sense of distribution in the domain $y\in B_{1/2}$. Moreover, applying convergences \eqref{cvg1}--\eqref{cvg3} to the inequality \eqref{EulerCA2} yields 
	\begin{equation}\label{CA6}
		\bigg| \int_{B_{1/2}} \big( \big|\tilde{h}^{\infty}\big|^2 + 2 \tilde{p}^{\infty} \big) \tilde{h}^{\infty} \cdot \frac{y}{|y|} \bigg| \ge \delta_{2}.
	\end{equation}
	Since $\tilde{h}^{\infty}=\frac{1}{|y|}\tilde{\zeta}^{\infty}(\frac{y}{|y|})$ for some $\tilde{\zeta}^{\infty}\in \mathcal  W^{1,2}(\Sph^4)$, it follows by Proposition \ref{prop:HP} that there exists $\xi \in \mathcal W^{1,5/4}(\Sph^4)$ and a constant $p_0\in\R$ for which $\tilde{p}^{\infty}(y)=\frac{1}{|y|}\xi(\frac{y}{|y|}) + p_0$ holds for a.e. $y\in B_{1/2}$. Therefore, by the classification theorem of the Homogeneous solutions to Euler's equations, Theorem \ref{thm:Heuler}, there exists constant $\mB_0\in \R^3$ such that $|\tilde{h}^{\infty}|^2 + 2\tilde{p}^{\infty} = \mB_0$ for a.e. $y\in B_{1/2}$. Then by the divergence free property $\div \tilde{h}^{\infty} = 0$, we get
	\begin{align}\label{CA7}
		\int_{B_{1/2}} \big( |\tilde{h}^{\infty}|^2 + 2\tilde{p}^{\infty} \big) \tilde{h}^{\infty} \cdot \frac{y}{|y|} = \mB_0 \int_{B_{1/2}} \tilde{h}^{\infty} \cdot \frac{y}{|y|} = 0.
	\end{align}  
	This contradicts the inequality \eqref{CA6}. 
\end{proof}

\begin{corollary}\label{corol:cubicEst}
	Let $\phi\in \mC^{\infty}(\R^5)$ be a spherically symmetric function. Then for fixed constants $\delta_{1},\, \delta_{2}>0$, there exists $\ep>0$ such that if $(u,p)$ is a Leray-Hopf solution to the Navier-Stokes equation satisfying \eqref{h-A} for some $R>0$, then  
	\begin{equation*}
		\left\lvert \frac{1}{R^2} \int_{B_{R/2}} \big(|u|^2 + 2 p\big) u\cdot \nabla \phi  \right\rvert \le \delta_{1} + \delta_{2} \big( M[u](R) \big)^{\frac{3}{2}}.
	\end{equation*}
\end{corollary}
\begin{proof}
	The proof is almost exactly the same as that of Proposition \ref{prop:cubicEst}, except we replace $\frac{y}{|y|}$ by $\nabla \phi(y)$ in the inequality \eqref{CA6}--\eqref{CA7}. Since $\phi$ is spherically symmetric, there exists $\varphi(s)\in \mC^{\infty}\big([0,\infty)\big)$ such that $\phi(y)=\varphi(|y|)$. It follows that $\nabla \phi(y) = \frac{y}{|y|} \varphi^{\prime}(|y|)$, which is parallel to $\frac{y}{|y|}$. Thus the divergence free condition $\div \tilde{h}^{\infty}=0$ is used in the same way to show \eqref{CA7}, which leads to the contradiction.
\end{proof}

\begin{remark}
	For a suitable weak solution $(u,p)$, the following local energy inequality holds
	\begin{equation}\label{lei}
		\int |\nabla u|^2 \phi\le \int \frac{|u|^2}2\Delta \phi+(|u|^2+2p)u\cdot \nabla \phi,
	\end{equation}
	for all $\phi\in \mC_c^{\infty}(\R^5)$ with $\phi\ge 0$. Set $\phi$ to be a smooth spherically symmetric positive test function $\phi(x)=\phi(|x|)\ge 0$, with the properties $\phi = \frac{4}{R}$ in $B_{R/4}$ and $\phi = 0$ in $\R^5\backslash B_{R/2}$. Then it can be derived from \eqref{lei} that
	\begin{equation}
		M(R/4)\le C M(R/2)+\bigg|\frac{C}{R^2}\int_{B_{R/2}}\big(|u|^2+2p\big) \, u \cdot \frac{x}{|x|}\bigg|,
	\end{equation}
	where $M(R)$ is defined in \eqref{MR}. Under the assumption \eqref{h-A}, we can apply Proposition \ref{prop:cubicEst} on the above inequality to obtain the following inequality
	\begin{equation}
		M(R/4)\le C+ \ep C M^{\frac32}(R). 
	\end{equation}
\end{remark}


%
%

\section{An iteration argument}\label{sec:iteration}

\begin{lemma}\label{lemma:iter}
	Let $F(r)\vcentcolon (0,\infty)\to (0,\infty)$ be a positive function. Suppose there exists $\delta>0$ with
	\begin{equation*}
		\delta \le \min\left\{ \left\{F(1)\right\}^{-3/2}, 2^{-3/2} \right\},
	\end{equation*}
	such that the following recurrence inequality holds  
	\begin{equation}\label{recurr}
		F(4^{-m-1})\le 1+\delta \left\{F(4^{-m})\right\}^{3/2} \qquad \text{for all } \ m\in\mathbb{N}.
	\end{equation}
	Then $F$ is bounded by
	\begin{equation*}
		\sup\limits_{m\in\mathbb{N}} F\left(4^{-m}\right)
		\le \max\left\{2, \delta^{-2/3}\right\}.
	\end{equation*}
\end{lemma}
\begin{proof}
	For $s\in \mathbb N$, there are two cases 
	\[
	\left\{
	\begin{aligned}
		\text{Case A:}& \quad \delta \left\{F\left(4^{-s}\right)\right\}^{3/2}\le 1, \\
		\text{Case B:}& \quad \delta  \left\{F\left(4^{-s}\right)\right\}^{3/2}\ge 1. 
	\end{aligned}
	\right.
	\]
	Note that Case A exists since $\delta \left\{F(1)\right\}^{3/2} \le 1$. Moreover, for Case A we have that 
	\begin{equation}\label{caseA}
		F\left(4^{-s}\right) \le \delta^{-2/3}.
	\end{equation}
	For an arbitrary $m\in \N$, if Case A holds then we are done. Suppose otherwise that there exists an integer $s\in [1, m-1]$ for which Case A holds for $s$ and Case B holds for all integers in $[s+1, m]$. First, if $s=m-1$, then applying \eqref{recurr} on $F(4^{-m})$ and using the inequality \eqref{caseA} yields
	\begin{align}\label{base}
		F\left(4^{-m}\right) \le 1 + \delta \left\{ F\left(4^{1-m}\right) \right\}^{3/2} \le 1 + \delta \left\{ \delta^{-2/3} \right\}^{3/2} = 2.
	\end{align}
	
	If $s < m-1$, then we set $\ell\vcentcolon= m-s-1 \ge 1$. Moreover, we claim that
	\begin{align}\label{iter-claim}
		F\left(4^{-m}\right) \le 1 + \frac{1}{2} \left(2\delta\right)^{S(\ell)} \left\{ F\left(4^{\ell-m}\right) \right\}^{(3/2)^{\ell}} \quad \text{where } \ S(\ell) \vcentcolon= \sum\limits_{j=0}^{\ell-1}\left(\frac{3}{2}\right)^j.
	\end{align}
	We show the above inequality by induction. The base case $\ell=1$ is the same as the first inequality in \eqref{base}. For the inductive step, assume that $\ell>1$ and there is an integer $k \in [1,\ell)$ for which the following inequality holds
	\begin{equation}\label{step-k}
		F\left(4^{-m}\right) \le 1 + \frac{1}{2} \left(2\delta\right)^{S(k)} \left\{ F\left(4^{k-m}\right) \right\}^{(3/2)^k} \quad \text{where } \ S(k) \vcentcolon= \sum\limits_{j=0}^{k-1}\left(\frac{3}{2}\right)^j.
	\end{equation}
	Applying \eqref{recurr} on the term $F(4^{k-m})$ in the right hand side of the above, we get
	\begin{align*}
		F\left(4^{-m}\right) \le 1 + \frac{1}{2} \left(2\delta\right)^{S(k)} \left\{ 1 + \delta \left\{F\left(4^{k+1-m}\right)\right\}^{3/2} \right\}^{(3/2)^{k}}. 
	\end{align*}
	Since $m-k-1\in (s,m)$, Case B holds for $m-k-1$. Thus the above inequality yields
	\begin{align*}
		F\left(4^{-m}\right)
		\le& 1 + \frac{1}{2} \left(2\delta\right)^{S(k)}\left(2\delta\right)^{(3/2)^{k}} \left\{ F\left(4^{k+1-m}\right) \right\}^{(3/2)^{k+1}} \!\!= 1 + \frac{1}{2} \left(2\delta\right)^{S(k+1)} \left\{ F\left(4^{k+1-m}\right) \right\}^{(3/2)^{k+1}}\!\!. 
	\end{align*}
	This shows that \eqref{step-k} also holds for $k+1$, hence the claim \eqref{iter-claim} holds by induction.
	
	Next, we apply \eqref{recurr} on \eqref{iter-claim} once more, then using the fact that \eqref{caseA} holds for $s$, we obtain 
	\begin{align}\label{iter-temp}
		F\left(4^{-m}\right) \le& 1 + \frac{1}{2} \left(2\delta\right)^{S(l)} \left\{ F\left( 4^{-s-1} \right) \right\}^{(3/2)^l} \le 1 + \frac{1}{2} \left( 2\delta \right)^{S(\ell)} \left\{ 1 + \delta \left\{F\left( 4^{-s} \right)\right\}^{3/2} \right\}^{(3/2)^{\ell}}\\
		\le& 1 + \frac{1}{2} \left( 2\delta \right)^{S(\ell)} \left\{ 1+ \delta  \left\{\delta^{-2/3} \right\}^{3/2} \right\}^{(3/2)^{\ell}} 
		= 1 + \frac{1}{2} \left( 2\delta \right)^{S(\ell)} 2^{(3/2)^{\ell}}. \nonumber
	\end{align}
	Evaluating the Geometric series yields that
	\begin{equation*}
		S(\ell) = \sum_{j=0}^{\ell -1} (3/2)^j = 2 \left\{ \left(\frac{3}{2}\right)^{\ell}-1 \right\}.
	\end{equation*}
	Substituting the above into \eqref{iter-temp} yields the inequality 
	\begin{equation*}
		F\left(4^{-m}\right) 
		\le  1 + \frac{1}{2} \left( 2\delta \right)^{-2} \left\{ 2^{3} \delta^{2} \right\}^{(3/2)^{\ell}}.
	\end{equation*}
	By the assumption, we have $\delta \le 2^{-3/2}$. It holds that $2^{3} \delta^{2} \le 1 $. In addition, since $\ell = m-s -1\ge 1$, one has $(\frac{3}{2})^{\ell}\ge 1$. Using these inequalities in the above, we get 
	\begin{align*}
		F\left(4^{-m}\right) \le& 1 + \frac{1}{2} \left(2\delta\right)^{-2} \left\{2^{3} \delta^{2} \right\} \cdot \left\{2^{3} \delta^{2} \right\}^{-1+(3/2)^{\ell}} \le  1 + \frac{1}{2} \left(2\delta\right)^{-2} 2^{3} \delta^{2} = 2 . 
	\end{align*}
	This completes the proof.
\end{proof}

Rescaling Lemma \ref{lemma:iter} we get 

\begin{lemma}\label{lemma:rIter}
	Fix $b>0$. Let $F(r)\vcentcolon (0,\infty)\to (0,\infty)$ be a positive function. Suppose there exists $\delta$ with
	\begin{equation}\label{delta-b}
		\delta \le \min\left\{ \frac{b}{\{F(1)\}^{3/2}}, \frac{1}{2 \sqrt{2b}} \right\}
	\end{equation}
	such that the following recurrence inequality holds  
	\begin{equation}\label{recurr-b}
		F(4^{-m-1})\le b +\delta \left\{F(4^{-m})\right\}^{3/2} \qquad \text{for all } \ m\in\mathbb{N}.
	\end{equation}
	Then $F$ satisfies the following uniform bound
	\begin{equation*}
		\sup\limits_{m\in\mathbb{N}} F\left(4^{-m}\right)\le \max\left\{2b, \left(\frac{b}{\delta}\right)^{2/3} \right\}.
	\end{equation*}
\end{lemma}
\begin{proof}
	Define $\tilde{F}(r)\vcentcolon= b^{-1} F(r)$ and $\tilde{\delta}\vcentcolon=\delta \sqrt{b}$ then \eqref{recurr-b} is rewritten as
	\begin{equation*}
		\tilde{F}(4^{-m-1}) \le 1 + \tilde{\delta} \left\{ \tilde{F}(4^{-m}) \right\}^{3/2} \qquad \text{for all } \ m\in\mathbb{N}.
	\end{equation*}
	The condition \eqref{delta-b} yields that $\tilde{\delta}$ satisfies
	\begin{equation}\label{scaled-delta}
		\tilde{\delta} \le \left\{\tilde{F}(1)\right\}^{-3/2} \quad \text{ and } \quad \tilde{\delta} \le 2^{-3/2}.  
	\end{equation}
	Thus we can apply Lemma \ref{lemma:iter} on the pair $(\tilde{\delta},\tilde{F})$ to obtain that
	\begin{equation*}
		\tilde{F}\left(4^{-m}\right) \le \max \left\{ 2, \tilde{\delta}^{-2/3} \right\} \qquad \text{for all } \ m\in\N.
	\end{equation*}
	Substituting $\tilde{F}=F/b$ and $\tilde{\delta}= \delta \sqrt{b}$, we have
	\begin{equation}\label{Fbound-scaled}
		F\left( 4^{-m} \right) \le \max \left\{ 2b, \left(\frac{b}{\delta}\right)^{2/3} \right\} \qquad \text{for all } \ m\in\N.
	\end{equation}
	This completes the proof.
\end{proof}

%
%
\section{Boundedness of $M(R)$}\label{sec:sup-est}

\begin{proposition}\label{prop:fM}
	Let $C_0>0$ be a constant and $u$ 
	a suitable weak solution of \eqref{eq:problem}. For $\delta>0$, there exists $\ep_0>0$ depending on $C_0$ and $\delta$ such that if $(u,p)$ is a suitable weak solution satisfying 
	\begin{equation}
		\ep_0 \int_{B_R} \Big\{ \frac{|u|^2}{R^3} + \frac{|\nabla u|^2}{R}  \Big\} + \frac{1}{R^3} \int_{B_R} \left| u - \mP_R[u] \right|^2 + \frac{1}{R} \int_{B_R} \left| \nabla u - \nabla \mP_R[u] \right|^2 \le C_0 \ep_0
	\end{equation}
	for some $R>0$, then 
	\begin{equation*}
		\frac{1}{R^3} \int_{B_{R/4}} |u|^2 + \frac{1}{R} \int_{B_{R/4}} |\nabla u|^2 \le \delta 
	\end{equation*}
\end{proposition}
\begin{proof}
	We prove by contradiction. Let $\delta>0$ be fixed. Suppose otherwise. Then there exist a sequence of solutions $(u^k,p^k)$ and sequences of positive numbers $\{R_k\}_{k\in\mathbb{N}}$ and $\{ \ep_k \}_{k\in\mathbb{N}}$ such that
	\begin{subequations}\label{fM:CA'}
		\begin{gather}
			\ep_k \to 0  \qquad \text{as } \ k\to\infty,\label{fM:CA'1}\\
			\frac{1}{R_k^3} \int_{B_{R_k}} \left| u^k - \mP_{R_k}[u^k] \right|^2 + \frac{1}{R_k} \int_{B_{R_k}} \left| \nabla u^k - \nabla \mP_{R_k}[u^k] \right|^2 \le C_0 \ep_k,\label{fM:CA'2}\\
			\int_{ B_{R_k}} \Big\{ \frac{|u^k|^2}{R_k^3} + \frac{|\nabla u^k|^2}{R_k} \Big\} \le C_0, \qquad 
			\int_{ B_{R_k/4}} \Big\{ \frac{|u^k|^2}{R_k^3} + \frac{|\nabla u^k|^2}{R_k} \Big\} \ge  \delta.\label{fM:CA'3}
		\end{gather}
	\end{subequations}
	For each $k\in\mathbb{N}$ and $y\in B$, we define the scaled function
	\begin{gather*}
		\bar{u}^k(y) \vcentcolon= R_k u^{k}(R_k y),  \qquad \bar{h}^k \vcentcolon= R_k \mP_{R_k}[u^k](R_k y),\\
		\hat{p}^k(y) \vcentcolon= R_k^2 p^k(R_k y), \qquad \bar{p}^k(y) \vcentcolon= \hat{p}^k(y) - \left[ \hat{p}^k \right].
	\end{gather*} 
	Then inequalities \eqref{fM:CA'} yields that for all $k\in\mathbb{N}$
	\begin{subequations}\label{fM:CA}
		\begin{gather}
			 \int_{B} \left| \bar{u}^k - \bar{h}^k \right|^2 + \int_{B} \left| \nabla \bar{u}^k - \nabla \bar{h}^k \right|^2 \le C_0 \ep_k, \qquad \int_{B} \Big\{ |\bar{u}^k|^2 + |\nabla \bar{u}^k|^2 \Big\} \le C_0, \label{fM:CA1}\\
			\int_{ B_{1/4} } \Big\{ |\bar{u}^k|^2 + |\nabla \bar{u}^k|^2 \Big\} \ge  \delta,\label{fM:CA2}
		\end{gather}
	\end{subequations}
	and $(\bar{u}^k,\bar{p}^k)$ solves the following equations in the sense of distribution
	\begin{equation}\label{fM:distr}
		\div \bar{u}^k = 0, \qquad (\bar{u}^k\cdot \nabla)\bar{u}^k + \nabla \bar{p}^k = \Delta \bar{u}^k, \qquad -\Delta \bar{p}^k = \div \div \big( \bar{u}^k \otimes \bar{u}^k \big),
	\end{equation}
	From the second inequality of \eqref{fM:CA1} and Proposition \ref{prop:bogo}, we have
	\begin{equation}
		\sup\limits_{k\in\N} \left\| \bar{p}^k \right\|_{L^{5/3}(B)} \le C \sup\limits_{k\in\N} \left\| \bar{u}^k \right\|_{\mathcal W^{1,2}(B)} \left\{ 1 + \left\| \bar{u}^k \right\|_{\mathcal W^{1,2}(B)}  \right\} \le C \sqrt{C_0} \left\{ 1 + \sqrt{C_0} \right\}.
	\end{equation} 
	Thus there exist $p^{\infty}\in L^{5/3}(B)$ and a subsequence, which is still denoted as $\{ \bar{p}^k \}_{k\in\mathbb{N}}$ such that
	\begin{equation}\label{cvg-p}
		\bar{p}^k \rightharpoonup p^{\infty} \quad  \text{ weakly in } \ L^{5/3}(B) \ \text{ as } \ k\to \infty.
	\end{equation}
	The conditions \eqref{fM:CA1} is the same as \eqref{NSCA1} in \ref{case-1} for the proof of Proposition \ref{prop:cubicEst}. Thus by the same argument, we obtain that there exists a subsequence and a function $h^{\infty}\in \mH(1/2)$ such that
	\begin{equation}
			\label{cvg-u}\bar{u}^{k} \to h^{\infty} \text{ strongly in } L^3(B), \quad \bar{u}^{k} \rightharpoonup h^{\infty} \text{ weakly in } L^{10/3}(B), \quad \nabla \bar{u}^{k} \rightharpoonup \nabla h^{\infty} \text{ weakly in } L^{2}(B),
	\end{equation}
	as $k\to\infty$. By the convergences \eqref{cvg-p}--\eqref{cvg-u} and the equation \eqref{fM:distr}, it holds that $\left(h^{\infty},p^{\infty}\right)$ satisfies 
	\begin{equation*}
		\div h^{\infty} = 0, \qquad \left( h^{\infty}\cdot \nabla \right) h^{\infty} + \nabla p^{\infty} = \Delta h^{\infty}, \qquad -\Delta p^{\infty} = \div \div \left(h^{\infty}\otimes h^{\infty}\right) 
	\end{equation*}
	in the sense of distribution in $y\in B$. Since $h^{\infty}$ is homogeneous of degree $-1$, it follows by Proposition \ref{prop:HP} that there exists a constant $p_0\in \R$ such that $p^{\infty}-p_0$ is homogeneous of degree $-2$. By Sevrak's classification of homogeneous solution for Navier-Stokes equations \cite{Sverak}, it follows that $\left(h^{\infty}, p^{\infty} \right)= (0,p_0)$ in $y\in B_{1/2}$. 
	
	Let $\phi \in \mC_{c}^{\infty}(\R^5)$ be a positive spherically symmetric function such that $\phi =1$ in $B_{1/4}$ and $\phi= 0$ in $\R^{5}\backslash B_{1/2}$. Taking $\phi$ in the local energy inequality \eqref{lei} for $\left( \bar{u}^{k}, \bar{p}^{k} \right)$ and using the equation $\div \bar{u}^{k} =0$, we have 
	\begin{align*}
		\int_{B_{1/4}} \left| \nabla \bar{u}^{k} \right|^2 \le  C \int_{B_{1/2}}\left|\bar{u}^{k} \right|^2  + C \left|\int_{B_{1/2}} \Big\{ \frac{1}{2}\left| \bar{u}^{k} \right|^2 + \bar{p}^{k} \Big\} \bar{u}^k\cdot \nabla \phi   \right|.
	\end{align*}  
	Applying H\"older's inequality, Corollary \ref{corol:bogo-Scaled}, Proposition \ref{prop:PE} and condition \eqref{fM:CA1} on the above yields
	\begin{align*}
		\int_{B_{1/4}} \left| \nabla \bar{u}^k \right|^2 \le& C \left\| \bar{u}^k \right\|_{L^2(B_{1/2})}^2 + C \left\| \bar{u}^k \right\|_{L^3(B_{1/2})}^3 + C \left\| \bar{u}^k \right\|_{L^{5/2}(B_{1/2})} \left\| \hat{p}^k - \left[\hat{p}^k\right] \right\|_{L^{5/3}(B)} \\
		\le &  C \left\| \bar{u}^k \right\|_{L^2(B_{1/2})}^2 + C \left\| \bar{u}^k \right\|_{L^3(B_{1/2})}^3 + C \sqrt{C_0} \left(1+ \sqrt{C_0}\right) \left\| \bar{u}^k \right\|_{L^{5/2}(B_{1/2})} 
	\end{align*} 
	Therefore there exists some constant $C_1 >0$ depending on $C_0$ such that for all $k\in\mathbb{N}$
	\begin{equation*}
		\int_{B_{1/4}} \left\{ \left| \bar{u}^k \right|^2 + \left| \nabla \bar{u}^k \right|^2 \right\} \le C_1 \left\{ \left\| \bar{u}^k \right\|_{L^2(B_{1/2})}^2 + \left\| \bar{u}^k \right\|_{L^3(B_{1/2})}^3 + \left\| \bar{u}^k \right\|_{L^{5/2}(B_{1/2})} \right\}.
	\end{equation*}
	 By condition \eqref{fM:CA2}, it follows that for all $k\in\mathbb{N}$,
	 \begin{equation*}
	 	\delta \le C_1 \left\{ \left\| \bar{u}^k \right\|_{L^2(B_{1/2})}^2 + \left\| \bar{u}^k \right\|_{L^3(B_{1/2})}^3 + \left\| \bar{u}^k \right\|_{L^{5/2}(B_{1/2})} \right\}.
	 \end{equation*}
	 By convergence \eqref{cvg-u} and the fact that $\left(h^{\infty},p^{\infty}\right)=(0, p_0)$, we obtain
	 \begin{equation*}
	 	\delta \le C_1 \left\{ \left\| h^{\infty} \right\|_{L^2(B_{1/2})}^2 + \left\| h^{\infty} \right\|_{L^3(B_{1/2})}^3 + \left\| h^{\infty} \right\|_{L^{5/2}(B_{1/2})} \right\} =0.
	 \end{equation*}
	  This is a contradiction.
\end{proof}

The previous lemma, combined with lemmas \ref{lem:Lin} and \ref{lem:Linbbb} shows that 
if $\limsup_{R\to 0} M(R)<\infty$ and $u$ is close to a function 
$h\in \mathcal H(R)$ in $\mathcal H$ norm then 
$u$ is regular at zero. The aim of the next lemma is to show that a suitable lower bound on 
$m=\liminf_{R\to 0} M(R)<\infty$, implies that  $\limsup_{R\to 0} M(R)<\infty$.

Setting a spherically symmetric test function $\phi\in \mC_c^{\infty}(\R^5)$ in the local energy inequality \eqref{LEI}, we can obtain that for all $R>0$
\begin{equation}\label{M-energy}
	M(R) \le C_E M(4R) + \left|\frac{C_E}{R^2}\int_{B_{2R}} \left(|u|^2 + 2 p\right) u\cdot \nabla \phi \right|,
\end{equation}
where $C_E\ge 1$ is a constant depending only on the dimension. 

\begin{lemma}\label{lemma:MBound}
	Let $m=\liminf_{R\to 0} M(R)\in (8 C_E, \infty)$.
	There exists $\ep=\ep(m)>0$ which depends on $m$, such that if $(u,p)$ is a suitable weak solution to the Navier-Stokes equations satisfying
	\begin{equation}\label{h-assump}
		\frac{1}{R^3} \int_{B_R} \big|u- \mP_R[u] \big|^2 + \frac{1}{R} \int_{ B_R} \big|\nabla u - \nabla \mP_R[u] \big|^2 \le \ep M(R)
	\end{equation} 
	for all $R\in(0, 1]$, then the scaled function $M(R)$ is uniformly bounded
	\begin{align*}
		\sup\limits_{0<R\le 1} M(R) <\infty.
	\end{align*}
\end{lemma}

\begin{proof}
	By Corollary \ref{corol:cubicEst}, there exists $\ep>0$ such that if $(u,p)$ satisfies \eqref{h-assump} we get that
	\begin{align*}
		\bigg|\frac{C_E}{R^2}\int_{B_{2R}} \big( |u|^2 + 2 p \big) \, u \cdot \nabla \phi \bigg|  \le \delta_{1} + \delta_{2} \left(M(4R)\right)^{\frac{3}{2}}.
	\end{align*}
	Combining the above inequality with \eqref{M-energy} yields that there exists $\ep>0$ for which if $(u,p)$ satisfies \eqref{h-assump} with $\ep$ then for all $R>0$
	\begin{align}\label{M-recurr}
		M(4^{-1}R) \le C_E M(R) + \delta_{1} + \delta_{2} \left(M(R)\right)^{\frac{3}{2}} \le \delta_{1} + \left(M(R)\right)^{\frac{3}{2}} \left\{ \delta_{2}+ \frac{C_E}{\sqrt{M(R)}}\right\}.
	\end{align} 
	In order to apply the iteration inequality, Lemma \ref{lemma:rIter}, we set $F(4^{-m})\vcentcolon= M(4^{-m}R_0)$ for $m\in\mathbb{N}$. We need to check that for some small $R_0$ the following inequality holds
	\[
	\sup_{(0, R_0]}\left(\delta_2+\frac{C_E}{\sqrt{M(R)}}\right)\le \min\left\{ \frac{\delta_1}{\{F(1)\}^{3/2}}, \frac{1}{2 \sqrt{2\delta_1}} \right\}=\begin{cases}
	\frac{1}{2 \sqrt{2\delta_1}}, &\ \mbox{if}\ \sqrt{2\delta_1}\ge F(1),\\
	\frac{\delta_1}{\{F(1)\}^{3/2}}, &\ \mbox{if}\ \sqrt{2\delta_1}< F(1).
	\end{cases}
	\]
	Since $m = \liminf_{R \to 0^+} M(R) < \infty$, we can choose a point $0<R_0\le 1$ such that $m \le F(1)=M(R_0) \le 2m$ and $\inf_{R\in(0,R_0]}M(R) > \frac{m}{2}$. Then for all $R\in (0,R_0]$,
	\[
	\delta_2+\frac{C_E}{\sqrt{M(R)}}
	\le \delta_2 + \frac{\tilde{C}_E}{\sqrt{m}} \qquad \text{where } \ \tilde{C}_E := \sqrt{2} C_E.
	\]
	Thus we want to show that 
	for some choice of $\delta_1, \delta_2$, the following inequality is satisfied
	\[
	\delta_2+\frac{\tilde{C}_E}{\sqrt{m}}\le \begin{cases}
	\frac{1}{2 \sqrt{2\delta_1}}, &\ \mbox{if}\ \sqrt{2\delta_1}> F(1),\\
	\frac{\delta_1}{\{F(1)\}^{3/2}}, &\ \mbox{if}\ \sqrt{2\delta_1}\le F(1).
	\end{cases}
	\]
	 Take $\sqrt{2\delta_1}=m$. Since $ m \le F(1)\le  2m$, we have
	\[
	\frac{\delta_1}{\{2m\}^{3/2}} \le \frac{\delta_1}{\{F(1)\}^{3/2}} .
	\] 
	Thus it is enough to require 
	\[
	\delta_2+\frac{\tilde{C}_E}{\sqrt{m}}\le \frac{\delta_1}{\{2m\}^{3/2}}.
	\]
	By our choice $\delta_1=m^2/2.$ Hence the inequality that we demand is 
	\[
	\delta_2+\frac{\tilde{C}_E}{\sqrt{m}}\le \frac{\sqrt m}{4\sqrt 2}, 
	\]
	or equivalently we require that 
	\[
	0 \le m-4\sqrt 2 \sqrt m\delta_2- 8 C_E=(\sqrt m-2\sqrt2 \delta_2)^2-8\delta_2^2- 8 C_E.
	\]
	Thus the desired inequality is satisfied if we choose $\delta_2>0$ such that
	\[
	\sqrt m\ge 2\sqrt2 \delta_2+\sqrt{8(\delta_2^2 + C_E)}.
	\]
	Consequently, for $m$ as above we apply the iteration inequality Lemma \ref{lemma:rIter} to conclude that 
	\[
	\sup\limits_{k\in\mathbb{N}} F\left(4^{-k}\right)\le \max\left\{m^2, \left(\frac{m^2 \sqrt{m}}{2\delta_2\sqrt{m}+2\sqrt{2} C_E}\right)^{2/3} \right\}.
	\]
	
\end{proof}
\begin{remark}
	Let $m$ be as in Lemma \ref{lemma:MBound} and $\delta_2>0$ be the constant chosen in its proof. Then for large $m\in(8C_E, \infty)$
	the constant 
	 \begin{equation}\label{eq:Cmm}
	C_0(m)=\max\left\{m^2, \left(\frac{m^2 \sqrt{m}}{2\delta_2\sqrt{m}+2\sqrt{2} C_E}\right)^{2/3} \right\}
	\end{equation}
	is at least quadratically large. Consequently, in 
	Proposition \ref{prop:fM} one should take 
	$\ep$ sufficiently small. 
	
\end{remark}

\section{Regularity of solution}\label{sec:proofofmain}

\begin{lemma}\label{lemma:reg}
Suppose $m:=\liminf_{R\to 0}M(R)<\infty$. 
	If there exists a sufficiently small $\ep>0$, depending on $m$,  such that  $(u,P)$ is a suitable weak solution to the Navier-Stokes equations satisfying
	\begin{equation}\label{H-Assump}
		\ep \frac{1}{R^{3}}\int_{B_R} |p| + \frac{1}{R^3} \int_{B_R} \left| u - \mP_R[u] \right|^2 + \frac{1}{R} \int_{B_R} \left| \nabla u - \nabla \mP_R[u] \right|^2 \le \ep M[u](R), 
	\end{equation}
	for all $R\in(0,1]$ then $u$ is regular at $x=0$.
\end{lemma}
\begin{proof}
If $m\le 8C_E$ then we can apply Proposition \ref{prop:fM}, and 
hence the result follows.
	Now suppose $m\in(8C_E, \infty)$. In light of Lemma \ref{lemma:MBound}, there exists $\ep_1>0$ such that if $(u,P)$ satisfies \eqref{H-Assump} with $\ep\in (0,\ep_1)$ then $M(R)$ is uniformly bounded in $R\in(0,1]$ and we set 
	\begin{equation}\label{Mast}
		M_{\ast} \vcentcolon= \sup\limits_{0<R\le 1} M(R) \le 2^5 C_0(m), 
	\end{equation}
	where $C_0(m)$ is given by \eqref{eq:Cmm}.
	At this point we can apply Proposition \ref{prop:fM}
	with $C_0=2^5C_0(m)$ by choosing $\ep_0$ sufficiently small, however we can avoid this by using the monotonicity formula.

	For $0<R_1<R_2$, we have by Proposition \ref{prop:dAdr} that
	\begin{equation}\label{monot-eq}
		A(R_2) - A(R_1) = \int_{R_1}^{R_2} \frac{1}{r} \Big\{ D(r) + \frac{2}{r^2}\int_{B_r} \Big(\frac{|u|^2}{2} + P\Big) u \cdot \frac{x}{|x|} \Big\}\, d r
	\end{equation}
	By \eqref{Mast}, if $(u,P)$ satisfies \eqref{H-Assump} for some $\ep\le \ep_1$ then for all $R\in(0,1]$
	\begin{align*}
		|A(R)| =& \left|  \int_{B_R} \frac{1}{R^3} u\cdot \left( x\cdot \nabla  \right) u + \frac{9}{4r^3}|u|^2 - \frac{1}{R^2}\Big( \frac{|u|^2}{2} + P \Big) u \cdot \frac{x}{|x|} \right|\\
		\le & \frac{1}{R^2} \left(\int_{B_R} |u|^2 \right)^{1/2} \left(\int_{B_R} |\nabla u|^2 \right)^{1/2} + \frac{9}{4R^3} \int_{B_R} |u|^2 + \delta_1 + \delta_2 M_{\ast}^{3/2}\\
		\le & C M(R) + 1 + M_{\ast}^{3/2}  \le C M_{\ast} + 1 + M_{\ast}^{3/2} <\infty.   
	\end{align*}
	Thus we have the bound that
	\begin{equation*}
		\sup\limits_{0<R\le 1} |A(R)| < \infty
	\end{equation*}
	Taking the limit $R_1\to 0^+$, we get
	\begin{align}
		\lim\limits_{R_1\to 0^+}\left| \int_{R_1}^{R_2} \frac{1}{r} \Big\{ D(r) + \frac{2}{r^2}\int_{B_r} \Big(\frac{|u|^2}{2} + P\Big) u \cdot \frac{x}{|x|} \Big\}\, d r \right| \le A(R_0) + \sup\limits_{0<R\le 1} |A(R)| < \infty.
	\end{align}

	If $\liminf_{R\to 0}M(R)<\delta$, and $\delta$ is small  then by Lemmas \ref{lem:Lin} ans \ref{lem:Linbbb}, 
	$x=0$ is a regular point. Thus without loss of generality, we assume the case
	\begin{equation*}
		m\vcentcolon=\liminf_{R\to 0}M(R)\ge\delta>0.	
		\end{equation*}
	Then for small enough $R_2>0$, we have 
	\begin{equation}\label{temp:reg1}
		\inf\limits_{R\in(0,R_2]} M(R) \ge \frac{m}{2}.
	\end{equation}
	With this, we set the constants
	\begin{equation*}
		\delta_1\vcentcolon= \frac{1}{32} m, \qquad \delta_2 \vcentcolon= \frac{m}{32M_{\ast}^{3/2}}. 
	\end{equation*}

	By Proposition \ref{prop:cubicEst}, there exists $\ep_2>0$ such that if $(u,P)$ satisfies \eqref{H-Assump} with $\ep\le \min\{\ep_1,\ep_2\}$ then 
	 \begin{equation}\label{temp:reg2}
	 	\left| \frac{1}{R^2} \int_{B_R} \Big(\frac{|u|^2}{2} + P\Big) u\cdot \frac{x}{|x|}  \right| \le \frac{1}{32} m + \frac{m}{32M_{\ast}^{3/2}}\left\{M(R)\right\}^{3/2} \le \frac{m}{16}.
	 \end{equation}
	 Recall $D(r)$ defined in Proposition \ref{prop:dAdr}. By \eqref{temp:reg1}--\eqref{temp:reg2}, it holds that for all $R\in (0,R_2]$,
	 \begin{align*}
	 	& \left\{D(R) + \frac{2}{R^2}\int_{B_R} \Big(\frac{|u|^2}{2} + P\Big) u\cdot \frac{x}{|x|} \right\}\\
	 	= & \left\{  \int_{B_R} \Big\{ \frac{7}{2R^3}|u|^2 + \frac{3}{4 R^3} \left|\nabla\left(|x|u\right)\right|^2 + \frac{3(R^2-|x|^2)}{4R^3} |\nabla u|^2 \Big\} \right\}\\ 
	 	&+ \left\{ \frac{1}{4} M(R) + \frac{2}{R^2}\int_{B_R} \Big(\frac{|u|^2}{2} + P\Big) u\cdot \frac{x}{|x|} \right\} \\
	 	\ge & \left\{ \int_{B_R} \Big\{ \frac{7}{2R^3}|u|^2 + \frac{3}{4 R^3} \left|\nabla\left(|x|u\right)\right|^2 + \frac{3(R^2-|x|^2)}{4R^3} |\nabla u|^2 \Big\} \right\} +\frac{m}{8}  - \frac{m}{8} \\
	 	= & \left\{  \int_{B_R} \Big( \frac{7}{2R^3}|u|^2 + \frac{3}{4 R^3} \left|\nabla\left(|x|u\right)\right|^2 + \frac{3(R^2-|x|^2)}{4R^3} |\nabla u|^2 \Big) \right\} \ge 0.
	 \end{align*} 
	 
	 For a pair of numbers $0<s<S<1$, and a sequence of positive numbers $R_k\to 0$,  we have
	 from the scale invariance of $A$ 
	 \begin{align*}
	 A[u](SR_k)-A[u](sR_k)
	 &=
	 A[u^k](S)-A[u^k](s)\\
	& \ge 
	 \int_s^S  \int_{B_R} \Big( \frac{7}{2R^3}|u^k|^2 + \frac{3}{4 R^3} \left|\nabla\left(|x|u^k\right)\right|^2 + \frac{3(R^2-|x|^2)}{4R^3} |\nabla u^k|^2 \Big). 
	 \end{align*}
	 Since $A[u](R)$ is monotone and bounded, then $\lim_{R\to 0^+}A[u](R)$ exists.
	 Consequently, for fixed $s, S$ we have 
	 \[
	 \lim_{k\to \infty}(A[u](SR_k)-A[u](sR_k))=0.
	 \]
	 This and Fatou's lemma yield
	 \[
	 \int_s^S  \int_{B_R} \Big( \frac{7}{2R^3}|\bar u|^2 + \frac{3}{4 R^3} \left|\nabla\left(|x|\bar u\right)\right|^2 + \frac{3(R^2-|x|^2)}{4R^3} |\nabla \bar u|^2 \Big) =0, 
	 \]
	 where $\bar u$ is the limit in $\mathcal W^{1,2}(B_2)$, say,  of $u^k(x)=R_ku(R_kx)$, for some subsequence of $\{R_k\}$.
	 Hence, we infer that $\bar u\equiv 0$. It remains to show that this is in contradiction with \eqref{temp:reg1}.
	 
	 Indeed, \eqref{temp:reg1} implies that there is a sequence $R_k$ such that 
	 $\lim_{R_k\to 0} M(R_k)\ge \frac {M_\infty} 2$. Hence, for sufficiently large $k$ one has 
	 \[
	 \int_{B_{R_k}}\frac{|u|^2}{R^3_k}+\frac{|\nabla u|^2}{R_k}\ge \frac{M_\infty}3.
	 \]
	 Again,  we consider two scenarios: a) $\int_{B_{R_k}}\frac{|u|^2}{R^3_k}\ge \frac{M_\infty}6$
	 or b) $ \int_{B_{R_k}}\frac{|\nabla u|^2}{R_k}\ge \frac{M_\infty}6$. 
	 
	 For a) we have $\int_{B_1}|u^k|^2=\int_{B_{R_k}}\frac{|u|^2}{R^3_k}\ge \frac{M_\infty}6$.
	 This is a contradiction in view of the 
	  strong convergence $u^k\to 0$ in $L^2(B_1)$. As for b) we can use the 
	  weak energy inequality to finish the proof. 
\end{proof}

\appendix

%
%

\section{Lin's perturbation method}\label{sec:appendix}
Let $(v,q)(x,t)\to \R^{3}\times \R \to \R^3\times \R$ be a suitable weak solution to the time evolving equations with spatial dimension $N=3$
\begin{equation}
	\left\{ \begin{aligned}
		&v_t + (v\cdot \nabla ) v + \nabla p = \Delta v,\\
		&\div v = 0,
	\end{aligned} \right.\qquad \text{ for }  \ (x,t)\in \R^{3}\times \R.
\end{equation}
 In \cite{Lin}, it is shown that there exists universal constants $C_0>0$ and $\ep_0>0$ such that if 
\begin{equation*}
	\int_{-1}^0\! \int_{B} \left\{ |v|^3 + |q|^{3/2} \right\} \, d x d t \le \ep_0, 
\end{equation*} 
then $(x,t)=(0,0)$ is a regular point and for all $K\in(0,1)$,
\begin{equation*}
	\|v\|_{\mC^{\alpha}(Q_K)} \le C_0 \quad \text{for some } \ \alpha >0, 
\end{equation*}
where $Q_K\vcentcolon= \{ (x,t) \, \vert \, |x|\le K \ \text{ and } \ -K^2 \le t \le 0 \}$.

Using a compactness argument, similar to the one in the proof of Proposition \ref{prop:fM}, 
it is easy to check that if $M(\rho)$ is small then  so is

\[
\int_{B_\rho/2} |u|^3+|p|^{\frac32}.
\]

\begin{lemma}\label{lem:Lin}
If $u$ is a suitable weak solution and 
\[
\int_{B_1} |u|^3+|p|^{\frac32}<\ep^*
\]
for some sufficiently small $\ep^*$, 
then 
\begin{equation}\label{eq:Lin1}
\frac1{\theta^5}\int_{B_\theta}\frac{|u-[u]_\theta|^3}{\theta^{\alpha_0}}+\frac{|p-[p]_\theta|^\frac32}{\theta^{\alpha_0}}\le \frac12 \int_{B_1} |u|^3+|p|^{\frac32},
\end{equation}
for some positive $\theta$ and $\alpha_0\in (0, \frac12)$.
\end{lemma}
\begin{proof}
If \eqref{eq:Lin1} fails, then there would be a sequence of solutions 
$(u_i, p_i)$ such that $\int_{B_1}|u_i|^3+|p_i|^{\frac32}:=\epsilon_i\to 0$ but 
\eqref{eq:Lin1} is not valid. Introduce 
\[
\bar u_i=\frac{u_i}{\epsilon_i}, \quad \bar p_i=\frac{p_i}{\epsilon_i}, 
\]
then 
\begin{equation}\label{eq:Lin4}
\epsilon_i \bar u_i\cdot \nabla \bar u_i+\nabla \bar p_i=\Delta \bar u_i, 
\quad 
\frac1{\theta^5}\int_{B_\theta}\frac{|\bar u-[\bar u]_\theta|^3}{\theta^{\alpha_0}}
+
\frac{|\bar p-[\bar p]_\theta|^\frac32}{\theta^{\alpha_0}}> \frac12, 
\quad  \int_{B_1} |\bar u_i|^3+|\bar p_i|^{\frac32}\le 2. 
\end{equation}
From the local energy inequality $u\in \mathcal W^{1, 2}_{loc}(B_1).$
Moreover, the following equation is satisfied in distributional sense
\begin{equation}
\Delta \bar p_i=-\epsilon_i \frac{\partial^2 (\bar u^k\bar u^l)}{\partial x_l\partial x_k}, 
\quad \mbox {in}\ B_1.
\end{equation}
From the Poisson representation theorem we can write $\bar p_i=h_i+g_i$, 
where $h_i$ is harmonic in $B_1$, and 
\begin{equation}
\left\{
\begin{array}{lll}
\Delta g_i=-\epsilon_i \frac{\partial^2 (\bar u^k\bar u^l)}{\partial x_l\partial x_k}
\quad \mbox {in}\ B_{\frac23},\\ 
g_i=0 \quad \mbox {on}\ \partial B_{\frac23}.
\end{array}
\right.
\end{equation}
From the Calder\'on-Zygmund estimates 
$g_i$  is uniformly bounded in ${L^{5/3}(B_{2/3})}$.
Consequently, $h_i\in L^{3/2}(B_{2/3})$ uniformly, hence 
from the local estimates for the harmonic functions 
\begin{align}
\int_{B_\theta}|\bar p_i-[\bar p_i]_\theta|^{\frac32}\le 
\int_{B_\theta}| h_i-[h_i]_\theta|^{\frac32}
\int_{B_\theta}|g_i-[g_i]_\theta|^{\frac32}\\
\le C_0 \theta^5\theta^{3/2}+C_0\epsilon_i\int_{B_{2/3}}|\bar u_i|^3.
\end{align}
For a suitable subsequence $\bar u_i\to \bar u$ in $\mathcal W^{1, 2}(B_{2/3})$
and $\bar p_i\to \bar p$ strongly in ${L^{3/2}(B_{2/3})}$.
Consequently, for sufficiently large $i$, we have 
\begin{equation}\label{eq:Lin2}
\int_{B_\theta}|\bar p_i-[\bar p_i]_\theta|^{\frac32} \le C_0 \theta^5\theta^{3/2}
\end{equation}
Since the limit $\bar u$ solves the Stokes system, then it follows that 
$\bar u$ is H\"older continuous with, say, exponent $2\alpha_0$, and therefore 
$\int_{B_\theta}|\bar u-[\bar u]_\theta|^{\frac32}\le \frac14  \theta^5\theta^{\alpha_0}$.
From the strong convergence $\bar u_i\to \bar u$ in $L^3(B_{2/3})$, we infer that 
\begin{equation}\label{eq:Lin3}
\int_{B_\theta}|\bar u-[\bar u]_\theta|^{\frac32}\le \frac13  \theta^5\theta^{\alpha_0}.
\end{equation}
Combining \eqref{eq:Lin2} and \eqref{eq:Lin2} we get a contradiction with the 
second inequality in \eqref{eq:Lin4}. 
\end{proof}

\begin{lemma}\label{lem:Linbbb}
If 
\[
\int_{B_1} |u|^3+|p|^{\frac32}<\ep^*
\]
for some sufficiently small $\ep^*$, 
then $u$ is H\"older continuous in $B_{1/2}$.
\end{lemma}
\begin{proof}
For given $\theta$, as in  Lemma \ref{lem:Lin}, we let 
\begin{equation}
u_1(x)=\frac{u(\theta x)-[u]_\theta}{\theta^{\alpha_0}}, 
\quad 
p_1(x)=\theta^{1-\alpha_0/3}(p(\theta x)-[p]_\theta), 
\end{equation}
and, moreover, 
\begin{align}
\theta([u]_\theta+\theta^{\alpha_0/3}u_1)\cdot \nabla u_1+\nabla p_1=\Delta u_1 \quad \mbox{in}\ B_1. 
\end{align}
Applying Lemma \ref{lem:Lin}, we get 
\begin{align}\label{eq:Freed}
\int_{B_1}|u_1|^3+|p_1|^{3/2}\le \frac{\ep^*}2
\end{align}
Indeed, in the compactness argument that we employed in the proof, the only step that 
must be changed is the limiting equation, which in this case takes the form 
\begin{align}
\left\{
\begin{array}{lll}
U_0\cdot \nabla \bar u+\nabla \bar p=\Delta \bar u,\quad \mbox{in}\ B_1\\
\div \bar u=0, \quad \mbox{in}\ B_1
\end{array}
\right.
\end{align}
where $U_0=\lim_{i\to \infty} \theta[\bar u_i]_\theta$ is a constant vectorfield with $|U_0|\le 2$.
Applying the regularity theory for the Stokes system with a constant drift \cite{Galdi}, 
we again conclude that $\bar u$, the limit in the proof of this slightly modified version 
of Lemma \ref{lem:Lin} is regular as well.

 Summarizing, we obtain that \eqref{eq:Freed} implies 
 \begin{align}
 \frac1{\theta^5}\int_{B_\theta}\frac{|u_1-[u_1]_\theta|^3}{\theta^{\alpha_0}}
 +
 \frac{|p_1-[p_1]_\theta|^\frac32}{\theta^{\alpha_0}}\le \frac12 \int_{B_1} |u_1|^3+|p_1|^{\frac32}
 \leq \frac{\ep^*}4
 \end{align}
Iterating this this inequality yields, for small $R$, 
\[
R^5\int_{B_R}|u-[u]_R|^3\leq C_0\epsilon R^{\alpha_0}, 
\]
implying that $u$ is H\"older continuous in $x$.
\end{proof}

%
%
\section{Computation for homogeneous Euler's equations}\label{sec:app-Euler}
We give a quick computation that expresses the Euler equations 
in spherical coordinates for self-similar solutions, as in \eqref{eq:EulerSphere}. 
A more general computation for the Navier-Stokes system can be found in \cite{Sverak}. 
By a direct computation 
\begin{align*}
\nabla_{\R^n, x_j}V^i=-\frac{v^ix_j}{r^3}
+
\frac1{r}\nabla_{\R^n, j}v^i
-\frac{2f}{r^4}x^ix^j
+\frac1{r^2}x^i\nabla_{\R^n, j}f+\frac{f\delta_{ij}}{r^2}.
\end{align*}
The parts of the convective terms can be computed as follows 
\begin{align*}
v^j\nabla_{\R^n, j}V^i
&=
-\frac{v^jv^ix_j}{r^3}
+
\frac1{r}v_j\nabla_{\R^n, j}v^i
-\frac{2f}{r^4}v^jx^ix^j
+\frac1{r^2}x^iv^j\nabla_{\R^n, j}f+\frac{fv^i}{r^2}\\
&=
\frac1{r}v_j\nabla_{\R^n, j}v^i
+\frac1{r^2}x^iv^j\nabla_{\R^n, j}f
+\frac{fv^i}{r^2}\\
&=
\frac1{r^2}(v_j(\nabla_{\Sph^{n-1}}v^i)^j-|v|^2\sigma^i)
+\frac1{r^3}x^iv^j(\nabla_{\Sph^{n-1}}f)^j
+\frac{fv^i}{r^2}\\
&=
\frac1{r^2}(v_j(\nabla_{\Sph^{n-1}}v^i)^j-|v|^2\sigma^i)
+\frac1{r^2}\sigma^iv^j(\nabla_{\Sph^{n-1}}f)^j
+\frac{fv^i}{r^2}.
\end{align*}
On the other hand 
\begin{align*}
\sigma^jf\nabla_{\R^n, j}V^i
&=
-\frac{f v^i}{r^2}
+\frac{f}{r}(\sigma^j\nabla_{\R^n, j}v^i)
-\frac{2f^2 x^i}{r^3 }
+
\frac{f x^i}{r^2}(\sigma^j\nabla_{\R^n, j}f)
+\frac{f^2\sigma^i}{r^2}\\
&=-\frac{f v^i}{r^2}
+\frac{f}{r}(\sigma^j\nabla_{\R^n, j}v^i)
-\frac{f^2 x^i}{r^3 }\\
&=
-\frac{f v^i}{r^2}
-\frac{f^2 \sigma^i}{r^2 }, 
\end{align*}
where the last line follows from the observation 
$\sigma\cdot \nabla_{\Sph^{n-1}} v=0$. 
Combining, we obtain
\[
r (V\cdot \nabla_{\R^n})V= \frac1{r^2}(v_j(\nabla_{\Sph^{n-1}}v^i)^j-|v|^2\sigma^i)
+\frac1{r^2}\sigma^iv^j(\nabla_{\Sph^{n-1}}f)^j
-\frac{f^2 \sigma^i}{r^2 }.
\]

Hence for the tangential components 
\begin{equation}
(v\cdot \nabla_{\Sph^{n-1}} )v+\nabla_{\Sph^{n-1}} p=0, 
\end{equation}
 and for the normal component
 \begin{equation}
-|v|^2 +v\cdot \nabla_{\Sph^{n-1}} f-f^2-2p=0, 
\end{equation}

Introducing $H=|v|^2+f^2+2p$, we see that the equation for the normal component is 
\begin{equation}
v\cdot \nabla_{\Sph^{n-1}} f=H, 
\end{equation}
Finally note that 

\begin{align*}
v\cdot \nabla_{\Sph^{n-1}} H
&=
v\cdot (2v  \nabla_{\Sph^{n-1}} v+2f  \nabla_{\Sph^{n-1}} f+2 \nabla_{\Sph^{n-1}} p)\\
&=
v\cdot 2f  \nabla_{\Sph^{n-1}} f\\
&=
2fH.  
\end{align*}

\end{document}